\documentclass[]{article}
\usepackage{amssymb}
\usepackage{amsmath}
\usepackage{float}
\usepackage{subfigure}
\usepackage{caption}
\usepackage{tikz}
\usetikzlibrary{decorations.markings}

\usepackage{hyperref}
\hypersetup{hypertex=true,
            colorlinks=true,
            linkcolor=blue,
            anchorcolor=blue,
            citecolor=blue}

\usetikzlibrary{backgrounds}
\newtheorem{thm}{Theorem}[section]
\newtheorem{conj}{Conjecture}[section]
\newtheorem{coro}{Corollary}
\newtheorem{definition}{Definition}
\newtheorem{lem}[thm]{Lemma}

\newtheorem{remark}{Remark}

\usepackage{amsfonts}
\newenvironment{proof}{{\medbreak\noindent\it  Proof.}\,}{\hfill$\square$\medbreak}

\usepackage{geometry}
\usepackage{makecell}

\usetikzlibrary{shapes, arrows, calc, arrows.meta, fit, positioning} 
\tikzset{  
	-stealth,auto,node distance =1.5 cm and 1.3 cm, thick,
	state/.style ={circle, draw, inner sep=0.3pt}, 
	point/.style = {circle, draw, inner sep=0.18cm, fill, node contents={}},  
	el/.style = {inner sep=2.5pt, align=right, sloped}  
}  

\begin{document}


\title{Arc-disjoint out- and in-branchings in compositions of  digraphs}
\author{J. Bang-Jensen\thanks{Department of Mathematics and Computer Science, University of Southern Denmark,Odense Denmark.(email: jbj@imada.sdu.dk)}\and Y. Wang \thanks{Department of Mathematics and Computer Science, University of Southern Denmark,Odense Denmark and School of Mathematics, Shandong University, Jinan 250100, China. (email:yunwang@imada.sdu.dk,wangyun\_sdu@163.com)}}

\maketitle

\begin{abstract}
  An out-branching $B^+_u$ (in-branching $B^-_u$) in a digraph $D$  is a connected spanning subdigraph of $D$ in which every vertex except the vertex $u$, called the root,  has in-degree (out-degree)  one. A {\bf good $\mathbf{(u,v)}$-pair} in $D$ is a pair of branchings $B^+_u,B^-_v$ which have no arc in common.
  Thomassen proved that  is  NP-complete to decide if a digraph has any good pair. A digraph is {\bf semicomplete} if it has no pair of non adjacent vertices. A {\bf semicomplete composition} is any digraph $D$  which is obtained from a semicomplete digraph $S$  by substituting an arbitrary digraph $H_x$ for each vertex $x$ of $S$.

  Recently the authors of this paper gave   a complete classification of semicomplete digraphs which have a good $(u,v)$-pair,  where $u,v$ are prescribed vertices of $D$. They also gave a polynomial algorithm which for a given semicomplete digraph $D$ and vertices $u,v$ of $D$, either produces a good $(u,v)$-pair in $D$ or a certificate that $D$ has  such pair. In this paper we show how to use the result for semicomplete digraphs to completely  solve the  problem of deciding whether a given semicomplete composition $D$, has a good $(u,v)$-pair for given vertices $u,v$ of $D$. Our solution implies that the problem is polynomially solvable for all semicomplete compositions.  In particular our result implies that there is a polynomial algorithm for deciding whether a given quasi-transitive digraph $D$ has a good $(u,v)$-pair for given vertices $u,v$ of $D$. This confirms a conjecture of Bang-Jensen and Gutin from 1998.

\end{abstract}
\noindent{\bf Keywords:} \texttt{arc-disjoint subdigraphs; in-branchings; out-branchings; semicomplete digraph; polynomial algorithm; semicomplete composition; quasi-transitive digraph }

\section{Introduction}\label{intro}
Notation follows \cite{bang2009} so we only repeat a few definitions here (see also Section
\ref{sec:term}).
Let $D=(V,A)$ be a digraph. An {\bf out-tree} ({\bf in-tree}) is an oriented tree in which every vertex except one, called the {\bf root}, has in-degree (out-degree) one. An {\bf out-branching} ({\bf in-branching}) of $D$ is a spanning out-tree (in-tree) in $D$. For a subdigraph $H$ of $D$ and a vertex $s$ of $H$ we denote by $B_{s,H}^+$, (resp., $B_{s,H}^-$) an arbitrary out-branching (resp., in-branching) rooted at $s$ in $H$. To simplify the notation, we set $B_s^+=B_{s,D}^+$ and $B_s^-=B_{s,D}^-$.

A digraph $D$ is {\bf strong} if there exists a path from $x$ to $y$  in $D$ for every ordered pair of distinct vertices $x$, $y$ of $D$ and $D$ is {\bf $k$-arc-strong} if $D\setminus{}A^{\prime}$ is strong for every subset $A^{\prime} \subseteq A$ of size at most $k - 1$. For a subset $X$ of $V$, we denote by $D\left\langle X \right\rangle$ the subdigraph of $D$ induced by $X$. 

The following well-known theorem, due to Edmonds, provides a characterization for the existence of $k$ arc-disjoint out-branchings rooted at the same vertex. 
\begin{thm}\label{edmonds1973}\cite{edmonds1973} {\bf(Edmonds' Branching Theorem)} 
A directed multigraph $D = (V,A)$ with a special vertex $s$ has $k$ arc-disjoint out-branchings rooted at $s$ if and only if
\begin{equation}
  \label{kbranchcond}
  d^-(X) \geq k,\;\; \forall \;\; \emptyset\neq X \subseteq V - s.
  \end{equation}
\end{thm}

Note that, by Menger's Theorem, (\ref{kbranchcond}) is equivalent to the existence of $k$ arc-disjoint $(s,v)$-paths for every $v\in V-s$. Furthermore, (\ref{kbranchcond}) can be checked in polynomial time via maximum flow calculations, see e.g.,
\cite[Section 5.4]{bang2009}.
Lov\'asz \cite{lovaszJCT21} gave a constructive proof of Theorem \ref{edmonds1973} which implies the existence of a polynomial algorithm for constructing a set of $k$ arc-disjoint branchings from a given root when (\ref{kbranchcond}) is satisfied. 

A natural related problem is to ask for a characterization of  digraphs having an out-branching and an in-branching which are arc-disjoint. Such pair will be called \textbf{a good pair} in this paper and more precisely we call it a \textbf{good $(u,v)$-pair} if the roots $u$ and $v$ are specified. Thomassen showed (see \cite{bangJCT51} and \cite{bangJGT100}) that for general digraphs it is NP-complete to decide if a given digraph has a good pair. This makes it interesting to study classes of digraphs for which we can find a good pair or decide that none exists in polynomial time.
Such studies have been made in e.g. \cite{bangJCT51,bangJGT100,bangJGT20b,bangJGT42,bangJGTsub,bangC24,bercziIPL109,gutinDM343}. In particular it was shown in \cite{bangJGTsub} that there is a polynomial algorithm for deciding whether a given semicomplete digraph $D$ has a good $(u,v)$-pair for specified vertices $u,v$ of $D$.
In fact the following surprisingly simple characterization holds. We use $P_{x,y}$ to denote a path from $x$ to $y$ for two vertices $x,y$ of $D$. Such a path is also called an $(x,y)$-path. 

\begin{figure}[H]
\centering
\subfigure{\begin{minipage}[t]{0.15\linewidth}
\centering\begin{tikzpicture}[scale=0.8]
		\filldraw[black](0,9) circle (3pt)node[label=left:$u$](u){};
		\filldraw[black](0,8) circle (3pt)node[label=left:$v$](v){};
		\path[draw, line width=0.8pt] (u) edge (v);
	\end{tikzpicture}\caption*{(a)}\end{minipage}}
\subfigure{\begin{minipage}[t]{0.15\linewidth}
\centering\begin{tikzpicture}[scale=0.8]
		\filldraw[black](0,10) circle (3pt)node[label=left:$u$](u){};
		\filldraw[black](0,9) circle (3pt)node[](w){};
		\filldraw[black](0,8) circle (3pt)node[label=left:$v$](v){};
			
		\path[draw, line width=0.8pt] (u) edge[bend left=30] (v);
		\path[draw, line width=0.8pt] (u) edge (w) edge (v);
		\end{tikzpicture}\caption*{(b)}\end{minipage}}
	\subfigure{\begin{minipage}[t]{0.15\linewidth}
			\centering\begin{tikzpicture}[scale=0.8]
		\filldraw[black](0,10) circle (3pt)node[label=left:$u$](u){};
		\filldraw[black](0,9) circle (3pt)node[](w){};
		\filldraw[black](0,8) circle (3pt)node[label=left:$v$](v){};
		
		\path[draw, line width=0.8pt] (u) edge[bend left=30] (v);
		\path[draw, line width=0.8pt]  (v) edge[bend left=30] (u);
		\path[draw, line width=0.8pt] (u) edge (w) edge (v);
	\end{tikzpicture}\caption*{(c)}\end{minipage}}
\subfigure{\begin{minipage}[t]{0.15\linewidth}
		\centering\begin{tikzpicture}[scale=0.8]
		\filldraw[black](0,10) circle (3pt)node[label=left:$u$](a){};
		\filldraw[black](0,9) circle (3pt)node[](b){};
		\filldraw[black](0,8) circle (3pt)node[](c){};
		\filldraw[black](0,7) circle (3pt)node[label=left:$v$](d){};
			
		\path[draw, line width=0.8pt] (a) edge[bend left=30] (c);
		\path[draw, line width=0.8pt] (b) edge[bend left=30] (d);
		\path[draw, line width=0.8pt] (d) edge[bend right=50] (a);
		\path[draw, line width=0.8pt] (a) edge (b) edge (c) edge (d);
		\end{tikzpicture}\caption*{(d)}\end{minipage}}
\subfigure{\begin{minipage}[t]{0.15\linewidth}
	\centering\begin{tikzpicture}[scale=0.8]
		\filldraw[black](0,10) circle (3pt)node[label=left:$u$](a){};
	    \filldraw[black](0,9) circle (3pt)node[](b){};
		\filldraw[black](0,8) circle (3pt)node[](c){};
		\filldraw[black](0,7) circle (3pt)node[label=left:$v$](d){};
					
		\path[draw, line width=0.8pt] (a) edge[bend left=30] (c);
		\path[draw, line width=0.8pt] (b) edge[bend left=30] (d);
		\path[draw, line width=0.8pt] (d) edge[bend left=30] (b);
		\path[draw, line width=0.8pt] (d) edge[bend right=50] (a);
		\path[draw, line width=0.8pt] (a) edge (b) edge (c) edge (d);
		\end{tikzpicture}\caption*{(e)}\end{minipage}}
\subfigure{\begin{minipage}[t]{0.15\linewidth}
		\centering\begin{tikzpicture}[scale=0.8]
		\filldraw[black](0,10) circle (3pt)node[label=left:$u$](a){};
		\filldraw[black](0,9) circle (3pt)node[](b){};
		\filldraw[black](0,8) circle (3pt)node[](c){};
		\filldraw[black](0,7) circle (3pt)node[label=left:$v$](d){};
		
		\path[draw, line width=0.8pt] (a) edge[bend left=30] (c);
		\path[draw, line width=0.8pt] (b) edge[bend left=30] (d);
		\path[draw, line width=0.8pt] (c) edge[bend left=30] (a);
		\path[draw, line width=0.8pt] (d) edge[bend right=50] (a);
		\path[draw, line width=0.8pt] (a) edge (b) edge (c) edge (d);
		\end{tikzpicture}\caption*{(f)}\end{minipage}}
\caption{Small semicomplete digraphs without good $(u,v)$-pairs.}
\label{fig-SDexceptions}
\end{figure}

\begin{thm}\cite{bangJGTsub}
  \label{thm:SDbranchchar}
  Let $D=(V,A)$ be a semicomplete digraph  with $u,v\in V$ (possibly $u=v$). Then $D$ has a good $(u,v)$-pair if and only if it satisfies (i) and (ii) below.
  \begin{itemize}
  \item[(i)] For every choice of $z,w\in V$ there are arc-disjoint paths $P_{u,z},P_{w,v}$ in $D$.
  \item[(ii)] $D$ is not one of the digraphs in Figure  \ref{fig-SDexceptions}(b)-(f).
    \end{itemize}
  \end{thm}

  It is easy to see that (i) must hold if $D$ has a good $(u,v)$-pair
  and when $D$ has at least 5 vertices the theorem says that (i) is also sufficient. It was shown in \cite{bangJCT51} that one can decide the existence of arc-disjoint paths $P_{u,z},P_{w,v}$ for given (not necessarily distinct) vertices $u,v,w,z$ in a semicomplete digraph in polynomial time. Hence one can check condition (i) above in polynomial time. This combined with the proof of Theorem \ref{thm:SDbranchchar} in \cite{bangJGTsub} implies the following.

  \begin{thm}\cite{bangJGTsub}\label{thm:polalgSD}
    There exists a polynomial algorithm which given a semicomplete digraph $D=(V,A)$ and two (not necessarily distinct) vertices of $D$ such that $D$ is not one of the digraphs in Figure  \ref{fig-SDexceptions}(b)-(f), either outputs a good $(u,v)$-pair of $D$ or vertices  $z,w\in V$ such that  $D$ has no pair of arc-disjoint paths $P_{u,z},P_{w,v}$.
    \end{thm}

    A digraph $D=(V,A)$ is {\bf quasi-transitive} if the presence of the arcs $uv,vw$ implies that there is an arc between $u$ and $w$. Bang-Jensen and Huang obtained the following result on $(u,u)$-pairs in quasi-transitive digraphs.

    \begin{thm}\cite{bangJGT20b} \label{thm:qtduupair}
      \begin{itemize}
        
  \item[(1)] Every 2-arc-strong quasi-transitive digraph $D=(V,A)$ has a good $(u,u)$-pair  for every choice of $u\in V$.
  \item[(2)] There exists a polynomial algorithm for deciding whether a given quasi-transitive digraph $D=(V,A)$ has a good $(u,u)$-pair  for a given vertex $u\in V$.
    \end{itemize}
\end{thm}

A bit later Bang-Jensen and Gutin conjectured the following.

\begin{conj}\cite{bangJGT28}
  \label{conj:QTDpol}
  There exists a polynomial algorithm for deciding for a  given quasi-transitive digraph $D=(V,A)$ and two vertices $u,v$ of $V$ whether $D$ has a good $(u,v)$-pair.
\end{conj}

Let $D$ be a digraph with vertex set $\{v_i: i\in[n]\}$, and let $H_1,\ldots,H_n$ be digraphs which are pairwise vertex-disjoint. The \textbf{composition} $D[H_1,\ldots,H_n]$ is the digraph $Q$ with vertex set $V(H_1)\cup\cdots\cup V(H_n)$ and arc set $(\bigcup_{i=1}^nA(H_i))\cup\{h_ih_j:h_i\in V(H_i),h_j\in V(H_j), v_iv_j\in A(D)\}$. We say that a composition\\
$Q=D[H_1,\ldots,H_n]$ is a \textbf{semicomplete composition} (resp., a {\bf transitive composition}) if $D$ is semicomplete (resp., transitive).
Bang-Jensen and Huang proved the following recursive characterization of quasi-transitive digraphs.

\begin{thm}\label{QTchara}\cite{bangJGT20b}
	Let $Q$ be a quasi-transitive digraph. 
	
	(i) If $Q$ is strong, then there exists a strong semicomplete digraph $S$ with $s$ vertices and quasi-transitive digraphs $Q_1,Q_2,\ldots,Q_s$ such that each $Q_i$ is either a vertex or is non-strong and $Q=S[Q_1,Q_2,\ldots,Q_s]$.
	
	(ii) If $Q$ is not strong, then there exists a transitive oriented graph $T$ with $t$ vertices and strong quasi-transitive digraphs $Q_1,Q_2,\ldots,Q_t$ such that $Q=T[Q_1,Q_2,\ldots,Q_t]$.
      \end{thm}

      By this theorem, strong  quasi-transitive  digraphs form a subclass of the class of semicomplete compositions and non-strong
      quasi-transitive digraphs form a subclass of the class of transitive compositions.
Hence the following result by Gutin and Sun extends Theorem \ref{thm:qtduupair}.

  \begin{thm}\cite{gutinDM343}
    \label{thm:GutinSun}
    There exists a polynomial algorithm for deciding whether a given semicomplete composition $D=(V,A)$ has a good $(u,u)$-pair for a given vertex $u\in V$.
  \end{thm}

  In this paper we consider the existence of good $(u,v)$-pairs in compositions of strong semicomplete digraphs and transitive digraphs and give a complete classification of semicomplete and transitive compositions with no good $(u,v)$-pair for given vertices $u,v$. The classification for compositions of strong semicomplete digraphs, which can be found in Theorem \ref{mainthm} is too involved to be stated in this introduction. For transitive compositions, which includes all compositions of non-strong semicomplete digraphs (see Remark \ref{rem:nonstrongSD}), there is a simple classification which we give in Proposition \ref{Dtransitive}.
All our proofs are constructive and can be converted to polynomial algorithms.

  \begin{thm}
    \label{thm:polalgSDcompbr}
    There exists a polynomial algorithm which given a composition $D=S[H_1,\ldots{},H_s]$, where $S$ is either transitive or semicomplete,  and two vertices $u,v$, decides whether $D$ has a good $(u,v)$-pair and outputs such a pair when it exists.
  \end{thm}

  The following corollary of Theorems \ref{QTchara} and \ref{thm:polalgSDcompbr} shows that Conjecture \ref{conj:QTDpol} is true.
  \begin{coro}
    There exists a polynomial algorithm which given a quasi-transitive digraph $D$   and two vertices $u,v$, decides whether $D$ has a good $(u,v)$-pair and outputs such a pair when it exists.
    \end{coro}

    {\bf The rest of the paper is organized as follows:} We start out with Section \ref{sec:term} which contains some extra definitions and results that will be used in the paper, in particular Theorem \ref{SDgoodpair} which plays a central role in the proof of Theorem \ref{mainthm}. In Section \ref{sec:characterization} we  state the main result of the paper, Theorem \ref{mainthm} and introduce  some semicomplete compositions that do not have a good $(u,v)$-pair for given vertices $u,v$.
    In Section \ref{sec:almostgoodpairSD} we  show that if a semicomplete digraph $D$ on at least 5 vertices does not have a good $(u,v)$-pair, then one can still produce an out-branching from $u$ and an in-branching at $v$ which share only a well-structured set of arcs. Section \ref{sec:proof} is devoted to the proof of Theorem \ref{mainthm}. This is done by proving a number of structural lemmas which we then apply to obtain the proof. In Section \ref{sec:transitive} we characterize transitive compositions with good $(u,v)$-pairs and use this result and Theorem \ref{mainthm} to characterize quasi-transitive digraphs without a good $(u,v)$-pair. Finally in Section \ref{sec:remarks} we show that the complicated characterization in Theorem \ref{mainthm} cannot be simplified to something similar to Theorem \ref{thm:SDbranchchar}.

\section{Terminology and additional results}\label{sec:term}

Note that to simplify notation we shall sometimes write $x\in D$ and $xy\in D$ to denote that $x$ is a vertex of $D$, respectively, that $xy$ is an arc of $D$.

A strong component $X$ is {\bf initial (resp., terminal)} if $X$ has no in-coming (resp., out-going) arcs in $D$. Note that for a semicomplete digraph $D$, a vertex $v$ belongs to the initial (resp., terminal) component of $D$ if and only if it is the root of some out-branching (resp., in-branching) of $D$.

\begin{remark}\label{rem:nonstrongSD}
    It is easy to see that for a non-strong
  semicomplete digraph $D$ there is a unique ordering $D_1,\ldots{},D_t$, $t\geq 2$ of its strong components and that
  $D=TT_t[D_1,\ldots{},D_t]$, there $TT_t$ is the transitive tournament on $t$ vertices.
\end{remark}
 
    Let $D$ be a digraph. We use $D\left\langle X\right\rangle$ to denote the subdigraph induced by a vertex set $X$. Let $D-X = D\left\langle V\backslash X \right\rangle$. We often identify a subdigraph $H$ of $D$ with its vertex set $V(H)$. For example, we write $D-H$ and $v\in H$ instead of $D-V(H)$ and $v\in V(H)$. A cycle and a path in the paper always means a directed cycle and path.  Let $C_n$ and $P_n$  denote a cycle and a path  with $n$ vertices, respectively.

Let $Q=D[H_1,\ldots{},H_r]$ be a composition. For a vertex $v\in V(Q)$, we use the notation $H(v)$ to denote the digraph $H_i$ containing $v$ and use $v_D$ to denote the vertex in $D$ which $v$ corresponds to.
The following direct consequence of the definition of a composition will be used many times in the paper.

    \begin{remark}\label{Remark-copyofD}
	For a given composition $Q=D[H_1,\ldots,H_n]$, if we pick an arbitrary vertex in each $H_i$, then the digraph induced by these vertices is isomorphic to $D$.
\end{remark}

A digraph $D=(V,A)$ has a \textbf{strong arc decomposition} if $A$ can be partitioned into arc-disjoint subsets $A_1$ and $A_2$ such that both $(V,A_1)$ and $(V,A_2)$ are strong. Clearly if a digraph $D$ has a strong arc decomposition, then it has a good $(u,v)$-pair for every choice of $u,v\in V$.  Let $S_4$ be a digraph which  obtained from the complete digraph with four vertices by deleting the arcs of a cycle of length four. Bang-Jensen, Gutin and Yeo gave  the following characterization of semicomplete compositions with a strong arc decomposition.  

\begin{thm}\label{SCarcdecom}\cite{bangJGT95}
	Let $S$ be a strong semicomplete digraph on $s\geq 2$ vertices and let $H_1,\ldots,H_s$ be arbitrary digraphs. Then $Q=S[H_1,\ldots,H_s]$ has a strong arc decomposition if and only if $Q$ is 2-arc-strong and is not isomorphic to one of the following four digraphs: $S_4$, $C_3[\overline{K}_2,\overline{K}_2,\overline{K}_2]$, $C_3[\overline{K}_2,\overline{K}_2,\overline{K}_3]$, $C_3[\overline{K}_2,\overline{K}_2,P_2]$.
\end{thm}

The following theorem follows immediately from Theorems \ref{QTchara}  and  \ref{SCarcdecom}. 
\begin{thm}\label{QTarcdecom}\cite{bangJGT95}
	A quasi-transitive digraph $Q$ has a strong arc decomposition if and only if $Q$ is 2-arc-strong and is not isomorphic to one of the following four digraphs: $S_4$, $C_3[\overline{K}_2,\overline{K}_2,\overline{K}_2]$, $C_3[\overline{K}_2,\overline{K}_2,\overline{K}_3]$, $C_3[\overline{K}_2,\overline{K}_2,P_2]$.
\end{thm}
 
\begin{thm}\cite{bangC24}\label{2ArcS-SD}
	Every 2-arc-strong semicomplete digraph $D$ contains a good $(u,v)$-pair for every choice of $u, v \in V(D)$.
      \end{thm}

\begin{lem} \label{(D-X)->D}
	Let $D$ be a digraph and let $X$ be a subset of $V(D)$ such that every vertex of $X$ has both an in-neighbor and an out-neighbor in $D-X$. If $D-X$ has a good $(u,v)$-pair then $D$ has a good $(u,v)$-pair. 
\end{lem}

\begin{proof}
	Let $(O,I)$ be a good $(u,v)$-pair in $D-X$. By assumption, every vertex $x\in X$ has an out-neighbor $x_O$ and in-neighbor $x_I$ in $D-X$. Then $(O+\{x_Ix:x\in X\},I+\{xx_O:x\in X\})$ is a good $(u,v)$-pair in $D$.
\end{proof}

Note that if $D$ has a good $(u,v)$-pair with $u\neq v$ then $D$ contains two arc-disjoint $(u,v)$-paths. Thus the following holds.
\begin{remark}\label{Remark-uvatleast2}
	Let $D$ be a digraph and let $u,v$ be two distinct vertices. If there is a good $(u,v)$-pair in $D$, then $d_D^+(u)\geq2$ and $d_D^-(v)\geq2$.  
\end{remark}

\begin{lem}\label{OutbranPath}\cite{bangJGTsub}
	Let $D$ be a digraph and let $u,v$ be two vertices of $D$ such that there is an out-branching rooted at $u$ in $D$. Suppose that $D$ has no out-branching $B_{u,D}^+$ which is arc-disjoint from some $(u,v)$-path. Then there exists a partition $V_1,V_2$ of $V(D)$ such that $v\in V_1,u\in V_2$ and $d_D^+(V_2)=1$.
\end{lem}

Next we consider arc-connectivity of  compositions of digraphs.

\begin{lem}\label{Qkstr-Qprikstr}
	Let $D$ be a digraph of order $n\geq 2$ and let $H_1,\ldots,H_n$ be arbitrary digraphs. Suppose that $Q=D[H_1,\ldots,H_n]$. If $Q$ is $k$-arc-strong and $H_i$ is an independent set with $|H_i|\geq k+1$, then the digraph $Q^{\prime}$ obtained from $Q$ by deleting a vertex from $H_i$ is also $k$-arc-strong.
\end{lem}
\begin{proof}
	Let $v$ be the vertex we delete from $H_i$. If $Q^{\prime}$ is not $k$-arc-strong, then there exists a partition $X,\bar{X}$ such that there are at most $k-1$ arcs from $X$ to $\bar{X}$. As $Q$ is $k$-arc-strong, $v$ has an in-neighbor $v^-$ in $X$ and out-neighbor $v^+$ in $\bar{X}$. It follows by the fact that $H_i$ is an independent set that none of $v^-$ and $v^+$ belong to $H_i$. Further, for each $w\in H_i-v$, $v^-wv^+$ forms one $(v^-,v^+)$-path in $Q^{\prime}$. As $|H_i|\geq k+1$, there are at least $k$ such paths and hence there are at least $k$ arcs from $X$ to $\bar{X}$, a contradiction.
\end{proof}

\begin{lem}\label{Dstr-Qkstr}
	Let $D$ be a digraph of order $n\geq 2$ and let $H_1,\ldots,H_n$ be arbitrary digraphs. Suppose that $Q=D[H_1,\ldots,H_n]$. If $D$ is strong and $|H_i|\geq k$ for each $i\in[n]$, then the digraph $Q$ is $k$-arc-strong.
\end{lem}
\begin{proof}
It suffices to show that for any two vertices $u$ and $v$ of $Q$, there are $k$ arc-disjoint $(u,v)$-paths in $Q$. If $u$ and $v$ belong to different $H_r$s, we may assume that both $u$ and $v$ are vertices of $D$ by Remark \ref{Remark-copyofD}, that is, $u_D=u,v_D=v$. As $D$ is strong, there is a $(u,v)$-path in $D$, say $uw_1\cdots w_tv$. Recall that $H(x)$ is the digraph $H_i$ containing $x$. For each $i\in[t]$, let $w_{ij}$ $(j\in[k])$ be $k$ distinct vertices of $H(w_i)$. Then $P_1,\ldots,P_k$, where  $P_j=uw_{1j}\cdots w_{tj}v$ $(j\in[k])$ form $k$ arc-disjoint $(u,v)$-paths in $Q$.

So assume that $u$ and $v$ belong to the same $H_r$, w.l.o.g. assume that $u_D=v_D=u$ by Remark \ref{Remark-copyofD}. Let $z^{\prime}$ be an out-neighbor of $u$ in $D$ and let $z\in H(z^{\prime})$ be arbitrary. By the above argument, there are $k$ arc-disjoint $(z,v)$-paths in $Q$, say $zw_{1j}\dots w_{tj}v, j\in[k]$. Let $z_1,\ldots, z_k$ be $k$ distinct vertices in $H(z^{\prime})$. Then $uz_jw_{1j}\cdots w_{tj}v$ $(j\in[k])$ are $k$ arc-disjoint $(u,v)$-paths in $Q$, which completes the proof.
\end{proof}

The following  alternative classification of semicomplete digraphs with good $(u,v)$-pairs is  more complicated than the  classification  given in Theorem \ref{thm:SDbranchchar} but it turns out to be much more useful in classifying semicomplete compositions with good $(u,v)$-pairs.

\begin{thm}\label{SDgoodpair}\cite{bangJGTsub}
Let $D$ be a semicomplete digraph and $u,v$ be arbitrary chosen vertices (possibly $u=v$). Then $D$ has a good $(u,v)$-pair if and only if $(D,u,v)$ satisfies none of the following conditions.

(i) $D$ is isomorphic to one of the digraphs in Figure \ref{fig-SDexceptions}.

(ii) $D$ is non-strong and either $u$ is not in the initial component of $D$ or  $v$ is not in the terminal component of $D$.

(iii) $D$ is strong and there exists an arc $e\in A(D)$ such that $u$ is not in the initial component of $D-e$ and $v$ is not in the terminal component of $D-e$.

(iv) $D$ is strong and there exists a partition $V_1,\ldots,V_{2\alpha+3}$ of $V(D)$ for some $\alpha\geq 1$ such that $v\in V_2, u\in V_{2\alpha+2}$ and all arcs between $V_i$ and $V_j$ with $i<j$ from $V_i$ to $V_j$ with the following exceptions. There exists precisely one arc from $V_{i+2}$ to $V_{i}$ for all $i\in[2\alpha+1]$  and it goes from the terminal component of $D\left\langle  V_{i+2}\right\rangle$ to the initial component of $D\left\langle  V_i\right\rangle$.
\end{thm}


\section{Formulating the main theorem}\label{sec:characterization}

 A tournament is \textbf{reverse-path transitive} if it is obtained from a transitive tournament by reversing the arcs of the unique hamiltonian path. We use $RT_n$ to denote a reverse-path transitive tournament with $n$ vertices.  Note that $RT_3=C_3$. Let $\overline{K}_r$ denote the digraph on $r$ vertices and no arcs.

\begin{definition}\label{typeAB}
Let $S$ be a semicomplete digraph and let $a,b$ be two arbitrary vertices (possibly $a=b$). The 3-tuple $(S,a,b)$ is said to be of

\textbf{type $A$}, for some $\alpha\geq 1$, if there exists a partition $V_1,\ldots, V_{2\alpha+1}$ of $V(S)$ such that $b\in V_2,a\in V_{2\alpha}$ and all arcs between $V_i$ and $V_j$ with $i<j$ from $V_i$ to $V_j$ with the following exceptions. There exists precisely one arc $x_iy_i$ from $V_{2\alpha+2-i}$ to $V_{2\alpha-i}$ for all $i\in[2\alpha-1]$ and it goes from the terminal component of $S\left\langle  V_{2\alpha+2-i}\right\rangle$ to the initial component of $S\left\langle  V_{2\alpha-i}\right\rangle$.

\textbf{type $B$}, for some $\beta\geq 1$, if there exists a partition $V_1,\ldots, V_{\beta+1}$ of $V(S)$ such that $b\in V_1,a\in V_{\beta+1}$ and all arcs between $V_i$ and $V_j$ with $i<j$ from $V_i$ to $V_j$ with the following exceptions. There exists precisely one arc $x_iy_i$ from $V_{\beta+2-i}$ to $V_{\beta+1-i}$ for all $i\in[\beta]$ such that there are two arc-disjoint $(y_{i-1},x_i)$-paths in $S\left\langle V_{\beta+2-i}\right\rangle$ when $i\geq 2$ and $y_{i-1}\neq x_{i}$. Further, $x_1$ belongs to the terminal component of $S\left\langle  V_{\beta+1}\right\rangle$ and $y_{\beta}$ belongs to the initial component of $S\left\langle  V_{1}\right\rangle$.
\end{definition}

\begin{figure}[H]
	\subfigure{\begin{minipage}[t]{0.48\linewidth}
			\centering\begin{tikzpicture}[scale=0.35]
				\foreach \i in {(5,0),(0,0),(-5,0)}{\draw[ line width=0.8pt] \i ellipse [x radius=50pt, y radius=70pt];}
				\coordinate [label=center:$V_1$] () at (-5,4);
				\coordinate [label=center:$V_2$] () at (-0,4);
				\coordinate [label=center:$V_3$] () at (5,4);
                                \draw[-stealth,line width=1.8pt] (-2,5) -- (2,5);
				\filldraw[black](0,1) circle (3pt)node[label=above:$a$](u){};
				\filldraw[black](0,-1.5) circle (3pt)node[label=above:$b$](v){};
				\filldraw[black](-5,-1.5) circle (3pt)node[label=above:$y_1$](y1){};
				\filldraw[black](5,-1.5) circle (3pt)node[label=above:$x_1$](x1){};
				
				\path[draw, line width=0.8pt] (x1) edge[bend left=30] (y1);
			\end{tikzpicture}\caption*{(a) type A with $\alpha=1$}\end{minipage}}~~~~
	\subfigure{\begin{minipage}[t]{0.48\linewidth}
			\centering\begin{tikzpicture}[scale=0.35]
				\foreach \i in {(-5,0),(0,0),(5,0)}{\draw[ line width=0.8pt] \i ellipse [x radius=50pt, y radius=70pt];}
				\coordinate [label=center:$V_1$] () at (-5,4);
				\coordinate [label=center:$V_2$] () at (0,4);
				\coordinate [label=center:$V_3$] () at (5,4);
                                \draw[-stealth,line width=1.8pt] (-2,5) -- (2,5);
				\filldraw[black](5,1) circle (3pt)node[label=above:$a$](u){};
				\filldraw[black](-5,1) circle (3pt)node[label=above:$b$](v){};
				\filldraw[black](-1,-1.5) circle (3pt)node[label=above:$y_1$](y1){};
				\filldraw[black](-5.5,-1.5) circle (3pt)node[label=above:$y_2$](y2){};
				\filldraw[black](5.5,-1.5) circle (3pt)node[label=above:$x_1$](x1){};
				\filldraw[black](1,-1.5) circle (3pt)node[label=above:$x_2$](x2){};
				
				\foreach \i/\j in {x1/y1,x2/y2}{\path[draw, line width=0.8pt] (\i) edge[bend left=35] (\j);}
			\end{tikzpicture}
			\caption*{(c) type B with $\beta=2$}\end{minipage}}\\
	\subfigure{\begin{minipage}[t]{1\linewidth}
			\centering\begin{tikzpicture}[scale=0.35]
				\foreach \i in {(-10,0),(-5,0),(0,0),(5,0),(10,0)}{\draw[ line width=0.8pt] \i ellipse [x radius=50pt, y radius=70pt];}
				\coordinate [label=center:$V_1$] () at (-10,4);
				\coordinate [label=center:$V_2$] () at (-5,4);
				\coordinate [label=center:$V_3$] () at (0,4);
				\coordinate [label=center:$V_4$] () at (5,4);
				\coordinate [label=center:$V_5$] () at (10,4);
                                \draw[-stealth,line width=1.8pt] (-2,5) -- (2,5);
				\filldraw[black](5,1) circle (3pt)node[label=above:$a$](u){};
				\filldraw[black](-5,1) circle (3pt)node[label=above:$b$](v){};
				\filldraw[black](-1,-1.5) circle (3pt)node[label=above:$y_1$](y1){};
				\filldraw[black](-5.5,-1.5) circle (3pt)node[label=above:$y_2$](y2){};
				\filldraw[black](-10.5,-1.5) circle (3pt)node[label=above:$y_3$](y3){};
				\filldraw[black](10.5,-1.5) circle (3pt)node[label=above:$x_1$](x1){};
				\filldraw[black](5.5,-1.5) circle (3pt)node[label=above:$x_2$](x2){};
				\filldraw[black](1,-1.5) circle (3pt)node[label=above:$x_3$](x3){};
				
				\foreach \i/\j in {x1/y1,x2/y2,x3/y3}{\path[draw, line width=0.8pt] (\i) edge[bend left=35] (\j);}
			\end{tikzpicture}\caption*{(c) type A with $\alpha=2$}\end{minipage}}
	\caption{Examples of 3-tuples $(S,a,b)$ of type A or type B. The bold arcs indicate that all possible arcs not shown from right  to left are present in the shown direction.}\label{fig6}
\end{figure}

The arcs $x_iy_i$ occurring in Definition \ref{typeAB} are  called  \textbf{backward} arcs of $S$.

\begin{table}[H]
	\caption{\textbf{A list of some semicomplete compositions without good $(u,v)$-pairs.} (See Figure \ref{fig-exceptions} for (a)-(e).) }
	\label{exceptions}
	\centering
	\begin{tabular}{ll}\hline
		(a) & $C_3[\overline{K}_2,\overline{K}_2,H]$ with $V(H)=\{u,v\}$ and $A(H)\subseteq\{vu\}$; \\ 
		(b) & $TT_3[u,\overline{K}_{n-2},v]$ or $TT_3[u,\overline{K}_{n-2},v]\cup\{vu\}$; \\
		(c) & $C_3[u,H,v]$, where $H$ is an arbitrary digraph with $|A(H)|\leq 1$; \\
		\makecell[l]{(d)\\~ } & \makecell[l]{$C_3[H(u),H(v),z]$ such that each vertex in $H(u)-u$ (resp., $H(v)-v$)\\ has exactly in-degree (resp., out-degree) one;} \\
		\makecell[l]{(e)\\~ } & \makecell[l]{$RT_4[\overline{K}_t,z,H(u),v]$ or $RT_5[H,\overline{K}_t,z,H(u),v]$, where $t\geq 1$, $H$ is an arbitrary\\ digraph and each vertex in $H(u)-u$ has exactly in-degree one;}\\
		\makecell[l]{(f)\\~ }& \makecell[l]{The strong semicomplete  composition  obtained from $RT_5[H,\overline{K}_1,z,H(u),v]$ by \\ reversing some arcs from $H$ to $\overline{K}_1$ or adding some arcs from $\overline{K}_1$ to $H$ (or both);}\\
		(g) & The digraph in Figure \ref{fig-SDexceptions} (e);\\
		\hline
	\end{tabular}
\end{table}
\begin{figure}[!ht]
	\centering
	\subfigure{\begin{minipage}[t]{0.45\linewidth}
			\centering\begin{tikzpicture}[scale=0.5]
				\draw[draw=black,line width=0.8pt] (-5,1.5) rectangle (-3,-1.5);\coordinate [label=center:$\overline{K}_2$] () at (-4,0);
				\draw[draw=black,line width=0.8pt] (-1,1.5) rectangle (1,-1.5);\coordinate [label=center:$\overline{K}_2$] () at (0,0);
				\draw[draw=black,line width=0.8pt] (3,1.5) rectangle (5,-1.5);\coordinate [label=center:$H$] () at (4,0.8);
				\filldraw[black](3.5,-0.2) circle (3pt)node[label=below:$u$](u){};
				\filldraw[black](4.5,-0.2) circle (3pt)node[label=below:$v$](v){};
				\filldraw[white](-4,1.6) circle (0.1pt)node[](a1){}; 
				\filldraw[white](-2.9,0) circle (0.1pt)node[](a2){}; 
				\filldraw[white](-1.1,0) circle (0.1pt)node[](b1){}; 
				\filldraw[white](1.1,0) circle (0.1pt)node[](b2){}; 
				\filldraw[white](4,1.6) circle (0.1pt)node[](c1){};
				\filldraw[white](2.9,0) circle (0.1pt)node[](c2){};
				
				\foreach \i/\j in {a2/b1,b2/c2}{\path[draw, line width=1.8pt] (\i) edge (\j);}
				\path[draw, red, dotted, line width=0.8pt] (v) edge (u);
				\path[draw, line width=1.8pt] (c1) edge[bend right=20] (a1);
			\end{tikzpicture}\caption*{(a)}\end{minipage}}
	\subfigure{\begin{minipage}[t]{0.45\linewidth}
			\centering\begin{tikzpicture}[scale=0.5]
				\draw[draw=black,line width=0.8pt] (-5,1.5) rectangle (-3,-1.5);\coordinate [label=center:$\overline{K}_{n-2}$] () at (-4,0);
				\filldraw[black](0,0) circle (3pt)node[label=below:$v$](v){};
				\filldraw[black](4,0) circle (3pt)node[label=below:$u$](u){}; 
				\filldraw[white](-2.9,0) circle (0.1pt)node[](a){};
				
				\path[draw, line width=1.8pt] (a2) edge (v);
				\path[draw, red, dotted, line width=0.8pt] (v) edge (u);
				\path[draw, line width=1.8pt] (u) edge[bend right=30] (a);
				\path[draw, line width=0.8pt] (u) edge[bend left=20] (v);
			\end{tikzpicture}\caption*{(b)}\end{minipage}}\\
	\subfigure{\begin{minipage}[t]{0.45\linewidth}
			\centering\begin{tikzpicture}[scale=0.5]
				\draw[draw=black,line width=0.8pt] (-5,1.5) rectangle (-3,-1.5);\coordinate [label=center:$H$] () at (-4,0.8);
				\filldraw[black](0,0) circle (3pt)node[label=below:$v$](v){};
				\filldraw[black](4,0) circle (3pt)node[label=below:$u$](u){}; 
				\filldraw[black](-4.5,0) circle (3pt)node(a1){};
				\filldraw[black](-3.5,0) circle (3pt)node(a2){};
				\filldraw[black](-4,-1) circle (3pt)node(a4){};
				\filldraw[white](-2.9,0) circle (0.1pt)node[](a3){};
				
				\path[draw, line width=1.8pt] (u) edge[bend right=30] (a);
				\path[draw, line width=1.8pt] (a3) edge (v);
				\path[draw, line width=0.8pt] (v) edge (u);
				\path[draw, red, dotted, line width=0.8pt] (a1) edge (a2);
			\end{tikzpicture}\caption*{(c)}\end{minipage}}
	\subfigure{\begin{minipage}[t]{0.45\linewidth}
			\centering\begin{tikzpicture}[scale=0.5]
				\draw[draw=black,line width=0.8pt] (-5,1.5) rectangle (-3,-1.5);\coordinate [label=center:$H(v)$] () at (-4,0.8);
				\draw[draw=black,line width=0.8pt] (3,1.5) rectangle (5,-1.5); \coordinate [label=center:$H(u)$] () at (4,-1);
				\filldraw[black](-4,-1) circle (3pt)node[label=right:$v$](v){};
				\filldraw[black](-4.5,0) circle (3pt)node(a1){};
				\filldraw[black](-3.5,0) circle (3pt)node(a2){};
				\filldraw[black](4.5,0) circle (3pt)node(a3){};
				\filldraw[white](-2.9,0) circle (0.1pt)node[](a4){};
				\filldraw[black](0,0) circle (3pt)node[label=below:$z$](z){};
				\filldraw[black](3.5,0) circle (3pt)node(b1){};
				\filldraw[black](4,1) circle (3pt)node[label=right:$u$](u){};
				\filldraw[white](2.9,0) circle (0.1pt)node[](b2){};
				
				\foreach \i/\j in {v/a2,v/a1,b1/u}{\path[draw, red, dotted, line width=0.8pt] (\i) edge (\j);}
				\foreach \i/\j in {a4/z,z/b2}{\path[draw, line width=1.8pt] (\i) edge (\j);}
				\path[draw, line width=1.8pt] (b2) edge[bend right=30] (a4);
			\end{tikzpicture}\caption*{(d)}\end{minipage}}\\
	\subfigure{\begin{minipage}[t]{0.5\linewidth}
			\centering\begin{tikzpicture}[scale=0.5]
				\draw[draw=black,line width=0.8pt] (-9,1.5) rectangle (-7,-1.5);\coordinate [label=center:$H$] () at (-8,0);
				\draw[draw=black,line width=0.8pt] (-5,1.5) rectangle (-3,-1.5);\coordinate [label=center:$\overline{K}_t$] () at (-4,0);
				\draw[draw=black,line width=0.8pt] (3,1.5) rectangle (5,-1.5); \coordinate [label=center:$H(u)$] () at (4,-1);
				\filldraw[white](-6.9,0) circle (0.1pt)node(a1){};
				\filldraw[white](-5.1,0) circle (0.1pt)node[](b1){};
				\filldraw[white](-2.9,0) circle (0.1pt)node[](b2){};
				\filldraw[black](0,0) circle (3pt)node[label=below:$z$](z){};
				\filldraw[black](3.5,0) circle (3pt)node[label=right:$u$](u){};
				\filldraw[black](4,1) circle (3pt)node(c1){};
				\filldraw[white](5.1,0) circle (0.1pt)node(c2){};
				\filldraw[white](2.9,0) circle (0.1pt)node(c3){};
				\filldraw[white](-4,2.5) circle (0.1pt)node(c4){};
				\filldraw[white](4,2.5) circle (0.1pt)node(c5){};
				\filldraw[black](8,0) circle (3pt)node[label=below:$v$](v){};
				
				\foreach \i/\j in {a1/b1,b2/z,z/c3,c2/v}{\path[draw, line width=1.8pt] (\i) edge (\j);}
				\path[draw, red, dotted, line width=0.8pt] (c1) edge (u);
				\path[draw, line width=2.4pt] (c5) edge (c4);
			\end{tikzpicture}\caption*{(e)}\end{minipage}}
%
		\caption{Semicomplete compositions with no good $(u,v)$-pair. The bold arcs indicate that all possible arcs are present in the shown direction, in particular, the bold arc in (e) from right to left indicates that all possible arcs between two non consecutive digraphs $H_i$s are from right to left.  The red dotted arcs indicate arcs that possibly not exist and the digraph $H$ is an arbitrary digraph (possibly empty). In (d)-(e), each vertex in $H(u)-u$ (resp., $H(v)-v$) has in-degree (resp., out-degree) one. The integer $t$ in (e) is at least one.}\label{fig-exceptions}
              \end{figure}

              \begin{lem}\label{specialexceptions}
	Let $Q$ be a composition of a semicomplete digraph and let $u,v$ be two vertices of $Q$. If $Q$ or $\overleftarrow{Q}$ is isomorphic to one of the digraphs in Table \ref{exceptions}, then there is no good $(u,v)$-pair in $Q$, where  $\overleftarrow{Q}$ is obtained from $Q$ be reversing all arcs and interchanging the names of $u$ and $v$.
\end{lem}

\begin{proof} Observe that by symmetry and Theorem \ref{SDgoodpair} (i), we only need to consider the case that $Q$ is isomorphic to one of the digraphs shown in Table \ref{exceptions} (a)-(f). Suppose that there is a good $(u,v)$-pair $(B_{u}^+,B_{v}^-)$ in $Q$.
	
	For the case that $Q$ is isomorphic to one of the digraphs in (a)-(c), let $P_1$ (resp., $P_2$) be the $(u,v)$-path in $B_{u}^+$ (resp., $B_{v}^-$). Let $u^+$ be the successor of $u$ on $P_1$ and let $v^-$ be the predecessor of $v$ on $P_2$. Observe that $P_1$ and $P_2$ are internally vertex-disjoint and $1\leq |A(P_i)|\leq 3$ for each $i$. If $u^+\neq v$, then the vertex $u^+$ needs an out-arc in $Q-P_1$ to ensure that $u^+$ can be collected in $B_{v}^-$. In the same way, if $v^-\neq u$, then $v^-$ needs an in-arc in $Q-P_2$ to ensure that $v^-$ can be collected in $B_{u}^+$.  Further, if both of the in- and out-arcs exist, then they should be distinct. However, it's impossible as either the only possible arc is $u^+v^-$ or  one of $u^+$ and $v^-$ has no such arc.
	
	Observe that for the digraphs in (d)-(f) the vertex $z$ is the only in-neighbor of every vertex in $H(u)-u$. So all arcs from $z$ to $H(u)-u$ must belong to the out-branching $B_{u}^+$. For the case that $Q=C_3[H(u),H(v),z]$ (the digraph in (d)), as $z$ is also the only out-neighbor of every vertex in $H(v)-v$, all arcs from $H(v)-v$ to $z$ must belong to $B_{v}^-$ and then $uvz$ and $zuv$ should be in the out- and in-branching of $Q$, respectively. Then we get a contradiction as the arc $uv$ is used twice in $(B_{u}^+,B_{v}^-)$. In the case that $Q$ is isomorphic to the digraph in (e) or (f), to collect the vertices of $\overline{K}_t$ to $B_{v}^-$, all arcs from $\overline{K}_t$ to $z$ must belong to the in-branching. Moreover, since all arcs from $z$ to $H(u)-u$ must belong to the out-branching $B_{u}^+$, we have $zu,uv\in B_{v}^-$. Then we obtain a contradiction again as either $uv$ or one of the  arcs from $\overline{K}_t$ to $z$ is used twice in $(B_{u}^+,B_{v}^-)$. 
\end{proof}

              Now we can state the main result of the paper. Recall that for any vertex $v\in V(Q)$, $v_S$ is the vertex in $S$ which  $v$  corresponds to.

\begin{thm}\label{mainthm}
	Let $S$ be a strong semicomplete digraph of order $s\geq 2$ and let $H_1,\ldots,H_s$ be arbitrary digraphs. Suppose that $Q=S[H_1,\ldots,H_s]$ and $u, v$ are two arbitrary vertices of $Q$ such that $d_{Q}^+(u)\geq 2$ and $d_{Q}^-(v)\geq 2$ if $u\neq v$. Then $Q$ has a good $(u,v)$-pair if and only if it satisfies none of the following conditions.
	
	(i) $Q$ or $\overleftarrow{Q}$ is isomorphic to one of the digraphs in Table \ref{exceptions}, where $\overleftarrow{Q}$ is obtained from $Q$ be reversing all arcs and interchanging the names of $u$ and $v$. 	
	
	(ii) $(S,u_S,v_S)$ is of type A and for each backward arc $xy$, either $|H(x)|=|H(y)|=1$ or, $d_Q^+(w)=1$ for every $w\in H(x)$ if $|H(x)|\geq 2$ and $d_Q^-(w)=1$ for every $w\in H(y)$ if $|H(y)|\geq 2$.
	
	(iii) $(S,u_S,v_S)$ is of type B and there exists a backward arc $xy$ such that either $|H(x)|=|H(y)|=1$ or, $d_Q^+(w)=1$ for every $w\in H(x)$ if $|H(x)|\geq 2$ and $d_Q^-(w)=1$ for every $w\in H(y)$ if $|H(y)|\geq 2$.
\end{thm}

\section{Almost good pairs in semicomplete digraphs}\label{sec:almostgoodpairSD}

In order to use Theorem \ref{SDgoodpair} in our proofs, we need the following refinement. Note that if $(S,a,b)$ is of type A (resp., type B), then any pair of branchings $B_{a}^+$ and $B_{b}^-$ in $S$ must share at least one backward arc (resp., all backward arcs) of $S$.

\begin{lem} \label{noGP}
	Let $S$ be a strong semicomplete digraph and let $a,b$ be two arbitrary vertices (possibly $a=b$). Suppose that $S$ has no good $(a,b)$-pair and it is not isomorphic to one of the digraphs in Figure \ref{fig-SDexceptions} (c)-(f) with $u=a,v=b$. Then one of the following statements holds.
	
	(I) $(S,a,b)$ is of type A and for each backward arc $x_ry_r$ $S$ has a pair of branchings $B_{a}^+,B_{b}^-$ such that $A(B_{a}^+)\cap A(B_{b}^-)=\{x_ry_r\}$.
	
	(II) $(S,a,b)$ is of type B and $S$ has  a pair of branchings $B_{a}^+,B_{b}^-$ such that the intersection of their arc sets is exactly the set of backward arcs of $S$, that is, $A(B_{a}^+)\cap A(B_{b}^-)=\{x_1y_1,\ldots,x_{\beta}y_{\beta}\}$.
\end{lem}

\begin{proof}
	Apply Theorem \ref{SDgoodpair} to the  strong semicomplete digraph $S$ with $u=a,v=b$. By the assumption that $S$ has no good $(a,b)$-pair and it is not isomorphic to one of the digraphs in Figure \ref{fig-SDexceptions} (c)-(f), we may assume that $(S,a,b)$ satisfies one of the statements (iii) and (iv) of Theorem \ref{SDgoodpair}.
	
	Suppose first that statement (iv) of Theorem \ref{SDgoodpair} holds. Then $(S,a,b)$ is of type A with $\alpha\geq 2$. We claim that for each backward arc $x_ry_r$, there exist branchings $B_{a}^+, B_{b}^-$ in $S$ such that they share only the arc $x_ry_r$, which implies (I). Let $V_1,\ldots, V_{2\alpha+1}$ be the partition of $V(S)$ and let $B$ be the set of backward arcs of $S$, that is, $B=\{x_1y_1,x_2y_2,\ldots,x_{2\alpha-1}y_{2\alpha-1}\}$.
	$$\mbox{Let }A=\{x_{2}y_{2},x_{4}y_{4},\ldots,x_{r}y_{r},x_{r+1}y_{r+1},x_{r+3}y_{r+3},\ldots, x_{2\alpha-1}y_{2\alpha-1}\}\mbox{ when } r \mbox{ is even }$$  $$\mbox{and }A=\{x_{2}y_{2},x_{4}y_{4},\ldots,x_{r-1}y_{r-1}, x_{r}y_{r},x_{r+2}y_{r+2},\ldots, x_{2\alpha-1}y_{2\alpha-1}\}\mbox{ when } r \mbox{ is odd,}$$ in particular,  $A=\{x_{1}y_{1},x_{3}y_{3},\ldots, x_{2\alpha-1}y_{2\alpha-1}\}$ when $r=1$ and  $A=\{x_2y_2,x_{3}y_{3},x_{5}y_{5},\ldots, x_{2\alpha-1}y_{2\alpha-1}\}$ when $r\in\{2,3\}$. Let $P_{a,y_{2\alpha-1}}$ be an $(a,y_{2\alpha-1})$-path containing an $(a,x_2)$-path in $V_{2\alpha}$ and all arcs in $A$ and let $P_{x_{1},b}$ be an $(x_{1},b)$-path containing a $(y_{2\alpha-2},b)$-path in $V_{2}$ and all arcs in $(B-A)\cup\{x_ry_r\}$.  Furthermore, we can construct these two paths (by linking subpaths and arcs in $A$ or $(B-A)\cup \{x_ry_r\}$ with some shortest $(y_i,x_{i+2})$-paths or arcs in $\{y_ix_{i+1}, x_{i+1}x_i,y_{i+1}y_i\}$) such that they only intersect in the arc $x_ry_r$. Now we can get a wanted pair $(B_{a}^+, B_{b}^-)$ as follows. Construct $B_{a}^+$ from an out-branching $B_{y_{2\alpha-1},S\left\langle  V_{1}\right\rangle}^+$ rooted at $y_{2\alpha-1}$ in $S\left\langle  V_{1}\right\rangle$ and the path $P_{a,y_{2\alpha-1}}$ by adding arcs $\{y_{2\alpha-1}z:z\in V(S)-V_1-V(P_{a,y_{2\alpha-1}})\}$ and, construct $B_{b}^-$ from an in-branching $B_{x_{1},S\left\langle  V_{2\alpha+1}\right\rangle}^-$ and the path $P_{x_{1},b}$ by adding arcs $\{zx_{1}:z\in V(S)-V_{2\alpha+1}-V(P_{x_1,b})-y_{2\alpha-1}\}$ and adding the arc $y_{2\alpha-1}a$ if $y_{2\alpha-1}\notin P_{x_1,b}$.

	It remains to consider the case that the statement (iii) of Theorem \ref{SDgoodpair} holds, that is, there exists an arc $xy\in A(S)$ such that $a\notin U_1$ and $b\notin U_t$, where  $U_1,\ldots,U_t$ is an acyclic ordering of the strong components of $S-xy$. Suppose first that there is an $(a,b)$-path $P_{a,b}$ in $S-xy$ (recall that possibly $a=b$), then clearly, $a,b\notin U_1\cup U_t$. We construct $B_{a}^+$ from an out-branching $B_{y,U_{1}}^+$ of $U_1$ and path $axy$ by adding arcs $\{yz:z\in V(S)-U_1-\{a,x\}\}$ and construct $B_{b}^-$ from an in-branching $B_{x,U_{t}}^-$ of $U_t$ and the path $xyP_{a,b}$ by adding arcs $\{zx:z\in V(S)-P_{a,b}-U_t-y\}$. Clearly, $B_{a}^+$ and $B_{b}^-$ are branchings and they share exactly one arc $xy$. Thus (I) holds with partition $V_1=V(U_1),V_3=V(U_t)$ and $V_2=V(S)-V_1-V_3$.
	
	Hence we may assume that there is no $(a,b)$-path in $S-xy$, which means that $a\in U_i,b\in U_j$ with $i>j$. Let $W_1=V(U_1\cup\cdots\cup U_j)$ and $W_2=V(S)-W_1$. Then $a\in W_2,b\in W_1$ and since $S$ is strong, $x$ belongs to the terminal component of $S\left\langle  W_2\right\rangle$ and $y$ belongs to the initial component of $S\left\langle  W_1\right\rangle$. It is not difficult to check that (II) holds with partition $W_1$ and $W_2$ if there is a pair of arc-disjoint $(y,b)$-path and out-branching rooted at $y$ in $S\left\langle W_1\right\rangle$ and,  a pair of arc-disjoint $(a,x)$-path and in-branching rooted at $x$ in $S\left\langle W_2\right\rangle$. 
	
	Suppose that $S\left\langle W_1\right\rangle$ has no pair of arc-disjoint $(y,b)$-path and an out-branching rooted at $y$ in $S\left\langle W_1\right\rangle$. By Lemma \ref{OutbranPath}, there is a partition $W,W_1-W$ of $W_1$ such that $b\in W,y\in W_1-W$ and only one arc from $W_1-W$ to $W$. Let $y^{\prime}$ be the head of the arc entering $W$. Again, if there is no arc-disjoint $(y^{\prime},b)$-path and out-branching rooted at $y^{\prime}$ in $S\left\langle  W\right\rangle$, there is a similar partition of $W$. By symmetry, we can apply a similar argument on $S\left\langle  W_2\right\rangle$. Repeating this, one can obtain a partition $V_1,\ldots,V_{\beta+1}$ of $V(S)$ with $b\in V_1, a\in V_{\beta+1}$ such that there exists precisely one arc $x_iy_i$ from $V_{\beta+2-i}$ to $V_{\beta+1-i}$ for all $i\in[\beta]$ and there is a pair of arc-disjoint $P_{y_{\beta},b}$ and $B_{y_{\beta},S\left\langle V_1\right\rangle}^+$ in $S\left\langle V_1\right\rangle$ and a pair of arc-disjoint $P_{a,x_1}$ and $B_{x_1,S\left\langle V_{\beta+1}\right\rangle}^+$ in $S\left\langle V_{\beta+1}\right\rangle$. It should be noted that we may further assume that there are two arc-disjoint $(y_{i-1},x_i)$-paths $P_{y_{i-1},x_i}$ and $P_{y_{i-1},x_i}^{\prime}$ in $S\left\langle V_{\beta+2-i}\right\rangle$ when $y_{i-1}\neq x_{i}$ since otherwise by Menger's theorem one can partition $V_{\beta+2-i}$ into $V$ and $V^{\prime}$ such that $y_{i-1}\in V, x_i\in V^{\prime}$ and there is only one arc from $V$ to $V^{\prime}$. This means we can assume that $(S,a,b)$ is of type B with partition $V_1,\ldots,V_{\beta+1}$.
	
	Then one can obtain a wanted branchings $B_{a}^+, B_{b}^-$ as follows. Construct $B_{a}^+$ from the path $P_{a,x_1}$, the out-branching $B_{y_{\beta},S\left\langle V_1\right\rangle}^+$ of $S\left\langle V_1\right\rangle$ and every path $P_{y_{i-1},x_i}$ with $y_{i-1}\neq x_i$ by adding all backward arcs $x_iy_i$ and all arcs from $y_{\beta}$ to uncovered vertices. In the same way, construct $B_{b}^-$ from $P_{y_{\beta},b}$, the  in-branching $B_{x_1,S\left\langle V_{\beta+1}\right\rangle}^+$ and every path $P_{y_{i-1},x_i}^{\prime}$ with $y_{i-1}\neq x_i$ by adding all backward arcs and all arcs from uncovered vertices to $x_{1}$, which implies (II).
\end{proof}

\section{Good pairs in compositions of strong semicomplete digraphs}\label{sec:proof}

Let $S$ be a strong semicomplete digraph of order $s\geq 2$ and let $H_1,\ldots,H_s$ be arbitrary digraphs. Suppose that $Q=S[H_1,\ldots,H_s]$ and $u, v$ are two arbitrary vertices of $Q$. Recall that  for a vertex $x\in Q$, $H(x)$ is the digraph $H_i$ containing the vertex $x$ and, $x_S$ is the vertex in $S$ which $x$ corresponds to. By Remark \ref{Remark-copyofD}, we assume that the following statement holds.
 
\begin{remark}\label{Remark-uSvS}
	If $u=v$ or $H(u)\neq H(v)$, then $u_S=u,v_S=v$ and, if $H(u)=H(v)$ and $u\neq v$, then $u_S=v_S=u$.
      \end{remark}

      \subsection{{\bf First case: $Q$  is 2-arc-strong}}
We start by giving a characterization of 2-arc-strong compositions $Q$ with a good $(u,v)$-pair. 

\begin{lem}\label{2arcstrong}
	Let $S$ be a strong semicomplete digraph on $s\geq 2$ vertices and let $H_1,\ldots,H_s$ be arbitrary digraphs. Suppose that $Q=S[H_1,\ldots,H_s]$ is 2-arc-strong and $u,v$ are two arbitrary vertices of $Q$. Then $Q$ has a good $(u,v)$-pair if and only if $Q$ is not isomorphic to the digraph  $C_3[\overline{K}_2,\overline{K}_2,H]$ with $V(H)=\{u,v\}$ and $A(H)\subseteq\{vu\}$, that is, one of the two digraphs in Table \ref{exceptions} (a).
\end{lem}
\begin{proof}
	The necessity follows by Lemma \ref{specialexceptions}. To see the sufficiency, observe that if $Q$ has a strong arc decomposition then $Q$ clearly has a good $(u,v)$-pair.  By Theorems \ref{SCarcdecom} and \ref{2ArcS-SD} and the fact that $S_4$ is  2-arc-strong, we may assume that $Q=C_3[\overline{K}_2,\overline{K}_2,H^{\prime}]$, where $H^{\prime}$ is isomorphic to one of the  digraphs $\overline{K}_2$, $\overline{K}_3$ or $P_2$. Let $x,y,z$ be the vertices of $C_3$ such that $C_3=xyzx$ and choose  $\{z_1,z_2\}\subseteq V(H(z))$, $\{x_1,x_2\}\subseteq V(H(x))$ and $\{y_1,y_2\}\subseteq V(H(y))$. W.l.o.g, assume that $u=z_1$ by Remark \ref{Remark-copyofD}.
	
	For the case that $u=v$ or $H(u)\neq H(v)$, we may assume that $v\in\{x_1,y_1,z_1\}$ by Remark \ref{Remark-copyofD}. Let $O=ux_2y_1z_2x_1y_2$ and $I=z_2x_2y_2u\cup C$, where $C=ux_1y_1u$. Then $(O,I-e)$ is a good $(u,v)$-pair in $Q\left\langle\{x_1,x_2,y_1,y_2,z_1,z_2\}\right\rangle$, where $e$ is the out-arc of $v$ in $C$. Lemma \ref{(D-X)->D} implies that $Q$ has a good $(u,v)$-pair.
	
	Hence it suffices to consider the case that $H(u)=H(v)$ and $u\neq v$, assume w.l.o.g that $z_2=v$. If $|V(Q)|=7$, that is, one of $H(x)$, $H(y)$ and $H(z)$ has order three, then construct a good $(u,v)$-pair $(O,I)$ in $Q$ as follows: Let $O^{\prime}=ux_1y_1vx_2y_2$. Further, let $O=O^{\prime}\cup y_2z_3$ and $I=ux_2y_1z_3x_1y_2v$ if $|H(z)|=3$,  and $O=O^{\prime}\cup vx_3$ and $I=x_2y_1ux_3y_2v\cup x_1y_2$ if $|H(x)|=3$, and $O=O^{\prime}\cup x_1y_3$ and $I=x_1y_2ux_2y_3v\cup y_1u$ if $|H(y)|=3$. See Figure \ref{fig-lemma41} (a)-(c).
	\begin{figure}[H]
		\subfigure{\begin{minipage}[t]{0.34\linewidth}
				\centering\begin{tikzpicture}[scale=0.3]
					\foreach \i in {(-3.46,-2),(3.46,-2),(0,4)}{\draw[ line width=0.8pt] \i ellipse [radius=75pt];}
					\filldraw[black](-1,4) circle (3pt)node[label=left:$u$](u){};
					\filldraw[black](1,4) circle (3pt)node[label=right:$v$](v){};
					\filldraw[black](0,5) circle (3pt)node[label=above:$z_3$](z3){};
					\filldraw[black](-5.19,-1) circle (3pt)node[label=below:$x_1$](x1){};
					\filldraw[black](-2.5,-3) circle (3pt)node[label=left:$x_2$](x2){};
					\filldraw[black](1.8,-3) circle (3pt)node[label=right:$y_1$](y1){};
					\filldraw[black](5.3,-1) circle (3pt)node[label=below:$y_2$](y2){};
					\foreach \i/\j in {x1/y2,x2/y1,u/x2,y1/z3,y2/v}{\path[draw, blue, dashed, line width=0.8pt] (\i) edge (\j);}
					\foreach \i/\j in {u/x1,v/x2,x1/y1,x2/y2,y1/v}{\path[draw, red, line width=1.5pt] (\i) edge (\j);}
					\path[draw, blue, dashed, line width=0.8pt] (z3) edge[bend right=50] (x1);
					\path[draw, red, line width=1.5pt] (y2) edge[bend right=50] (z3);
				\end{tikzpicture}\caption*{(a) $|H(z)|=3$}\end{minipage}}
		\subfigure{\begin{minipage}[t]{0.34\linewidth}
				\centering\begin{tikzpicture}[scale=0.3]
					\foreach \i in {(-3.46,-2),(3.46,-2),(0,4)}{\draw[ line width=0.8pt] \i ellipse [radius=75pt];}
					\filldraw[black](-1,4) circle (3pt)node[label=above:$u$](u){};
					\filldraw[black](1,4) circle (3pt)node[label=above:$v$](v){};
					\filldraw[black](-5.19,-1) circle (3pt)node[label=below:$x_1$](x1){};
					\filldraw[black](-1.73,-1.5) circle (3pt)node[label=below:$x_2$](x2){};
					\filldraw[black](-3.5,-3.2) circle (3pt)node[label=left:$x_3$](x3){};
					\filldraw[black](1.8,-3) circle (3pt)node[label=right:$y_1$](y1){};
					\filldraw[black](5.3,-1) circle (3pt)node[label=below:$y_2$](y2){};
					\foreach \i/\j in {x1/y2,u/x3,x2/y1,y1/u,y2/v}{\path[draw, blue, dashed, line width=0.8pt] (\i) edge (\j);}
					\foreach \i/\j in {u/x1,v/x3,v/x2,x1/y1,x2/y2,y1/v}{\path[draw, red, line width=1.5pt] (\i) edge (\j);}
					\path[draw, blue, dashed, line width=0.8pt] (x3) edge[bend right=50] (y2);
				\end{tikzpicture}\caption*{(b) $|H(x)|=3$}\end{minipage}}
			\subfigure{\begin{minipage}[t]{0.34\linewidth}
			\centering\begin{tikzpicture}[scale=0.3]
				\foreach \i in {(-3.46,-2),(3.46,-2),(0,4)}{\draw[ line width=0.8pt] \i ellipse [radius=75pt];}
				\filldraw[black](-1,4) circle (3pt)node[label=above:$u$](u){};
				\filldraw[black](1,4) circle (3pt)node[label=above:$v$](v){};
				\filldraw[black](-5.19,-1) circle (3pt)node[label=below:$x_1$](x1){};
				\filldraw[black](-1.73,-3) circle (3pt)node[label=left:$x_2$](x2){};
				\filldraw[black](1.73,-1.6) circle (3pt)node[label=below:$y_1$](y1){};
				\filldraw[black](5.19,-1) circle (3pt)node[label=below:$y_2$](y2){};
				\filldraw[black](3.5,-3.2) circle (3pt)node[label=right:$y_3$](y3){};
				\foreach \i/\j in {u/x2,x1/y2,y1/u,y2/u,x2/y3,y3/v}{\path[draw, blue, dashed, line width=0.8pt] (\i) edge (\j);}
				\foreach \i/\j in {u/x1,v/x2,x1/y1,x2/y2,y1/v}{\path[draw, red, line width=1.5pt] (\i) edge (\j);}
				\path[draw, red, line width=1.5pt] (x1) edge[bend right=55] (y3);
				\end{tikzpicture}\caption*{(c) $|H(y)|=3$}\end{minipage}}\\~~~~
		\subfigure{\begin{minipage}[t]{0.34\linewidth}
				\centering\begin{tikzpicture}[scale=0.3]
					\foreach \i in {(-3.46,-2),(3.46,-2),(0,4)}{\draw[ line width=0.8pt] \i ellipse [radius=75pt];}
					\filldraw[black](-1,4) circle (3pt)node[label=above:$u$](u){};
					\filldraw[black](1,4) circle (3pt)node[label=above:$v$](v){};
					\filldraw[black](-5.19,-1) circle (3pt)node[label=below:$x_1$](x1){};
					\filldraw[black](-1.73,-3) circle (3pt)node[label=left:$x_2$](x2){};
					\filldraw[black](1.73,-3) circle (3pt)node[label=right:$y_1$](y1){};
					\filldraw[black](5.19,-1) circle (3pt)node[label=below:$y_2$](y2){};
					\foreach \i/\j in {u/x2,x1/y2,x2/y1,y1/v,y2/u}{\path[draw, blue, dashed, line width=0.8pt] (\i) edge (\j);}
					\foreach \i/\j in {u/x1,u/v,v/x2,x1/y1,x2/y2}{\path[draw, red, line width=1.5pt] (\i) edge (\j);}
				\end{tikzpicture}\caption*{(d) $z_1z_2\in A(Q)$}\end{minipage}}
		\subfigure{\begin{minipage}[t]{0.34\linewidth}
				\centering\begin{tikzpicture}[scale=0.3]
					\foreach \i in {(-3.46,-2),(3.46,-2),(0,4)}{\draw[ line width=0.8pt] \i ellipse [radius=75pt];}
					\filldraw[black](-1,4) circle (3pt)node[label=above:$u$](u){};
					\filldraw[black](1,4) circle (3pt)node[label=above:$v$](v){};
					\filldraw[black](-5.19,-1) circle (3pt)node[label=below:$x_1$](x1){};
					\filldraw[black](-1.73,-3) circle (3pt)node[label=left:$x_2$](x2){};
					\filldraw[black](1.73,-3) circle (3pt)node[label=right:$y_1$](y1){};
					\filldraw[black](5.19,-1) circle (3pt)node[label=below:$y_2$](y2){};
					\foreach \i/\j in {u/x2,x1/y2,x2/y1,y1/v,y2/u}{\path[draw, blue, dashed, line width=0.8pt] (\i) edge (\j);}
					\foreach \i/\j in {u/x1,x1/x2,x1/y1,x2/y2,y2/v}{\path[draw, red, line width=1.5pt] (\i) edge (\j);}
				\end{tikzpicture}\caption*{(e) $x_1x_2\in A(Q)$}\end{minipage}}
		\subfigure{\begin{minipage}[t]{0.34\linewidth}
				\centering\begin{tikzpicture}[scale=0.3]
					\foreach \i in {(-3.46,-2),(3.46,-2),(0,4)}{\draw[ line width=0.8pt] \i ellipse [radius=75pt];}
					\filldraw[black](-1,4) circle (3pt)node[label=above:$u$](u){};
					\filldraw[black](1,4) circle (3pt)node[label=above:$v$](v){};
					\filldraw[black](-5.19,-1) circle (3pt)node[label=below:$x_1$](x1){};
					\filldraw[black](-1.73,-3) circle (3pt)node[label=left:$x_2$](x2){};
					\filldraw[black](1.73,-3) circle (3pt)node[label=right:$y_1$](y1){};
					\filldraw[black](5.19,-1) circle (3pt)node[label=below:$y_2$](y2){};
					\foreach \i/\j in {u/x2,x1/y2,x2/y1,y1/v,y2/u}{\path[draw, blue, dashed, line width=0.8pt] (\i) edge (\j);}
					\foreach \i/\j in {u/x1,v/x2,x1/y1,y1/y2,y2/v}{\path[draw, red, line width=1.5pt] (\i) edge (\j);}
				\end{tikzpicture}\caption*{(f) $y_1y_2\in A(Q)$}\end{minipage}}
			\caption{Figures in Lemma \ref{2arcstrong}. The red fat (resp., blue dashed) arcs indicates the out- (resp., in-) branching.}\label{fig-lemma41}
	\end{figure}
	
	Therefore, we may assume that $|H(x)|=|H(y)|=|H(z)|=2$. As $Q$ is not isomorphic to one of the exceptions in Table \ref{exceptions} (a), we may assume that one of the digraphs $H(x)$, $H(y)$ or $H(z)$ has an arc. Further, if there is no arc with both ends in $H(x)$ or $H(y)$, then $u$ dominates $v$, i.e., $z_1z_2\in A(Q)$ as we assume that $Q$ is not isomorphic to the exception. Now one can construct a good $(u,v)$-pair $(O,I)$ in $Q$ as follows. Let $I=x_1y_2ux_2y_1v$ and construct $O$ from $ux_1y_1$ by adding $uvx_2y_2$ if $uv\in A(Q)$ (resp., $x_1x_2y_2v$ if $x_1x_2\in A(Q)$, or $y_1y_2vx_2$ if $y_1y_2\in A(Q)$). See Figure \ref{fig-lemma41} (d)-(f). This completes the proof.	
      \end{proof}

      \subsection{{\bf The case when $S$ has a good $(u_S,v_S)$-pair}}

Next we consider the case that there is a good $(u_S,v_S)$-pair in $S$.

\begin{lem}\label{goodinD}
	Let $S$ be a strong semicomplete digraph on $s\geq 2$ vertices and let $H_1,\ldots,H_s$ be arbitrary digraphs.  Let $Q=S[H_1,\ldots,H_s]$ and let $u,v$ be two vertices of $Q$ such that either $u=v$ or $H(u)\neq H(v)$. If there is a good $(u_S,v_S)$-pair in $S$, then $Q$ has a good $(u,v)$-pair which uses no arc in $\cup_{i\in[s]}A(H_i)$.
\end{lem}
\begin{proof}
	By Remark \ref{Remark-uSvS}, we have $u_S=u, v_S=v$. Let $(O,I)$ be a good $(u,v)$-pair in $S$. Let $u_Iu$ be any in-arc of $u$ and $vv_O$ be any out-arc of $v$ in $S$. The arcs exist as $S$ is strong. Then one can construct the wanted $(u,v)$-pair $(B_{u}^+,B_{v}^-)$ of $Q$ as follows. Construct $B_{u}^+$ from $O$ by adding arcs $\{xr:xy\in A(O),r\in H(y)\}$ and $\{u_Ir:r\in H(u)-u\}$ and construct $B_{v}^-$ from $I$ by adding arcs $\{ry:xy\in A(I),r\in H(x)\}$ and $\{rv_O:r\in H(v)-v\}$.
	
	It should be noted that if $u_Iu$ is an arc of the  in-branching $I$ in $S$, then the vertices in $H(u_I)$ are collected by $u$ in $B_{v}^-$ and vertices in $H(u)-u$ are collected by $u_I$ in $B_{u}^+$. So no arcs between $H(u_I)$ and $H(u)$ are used both in $B_{v}^-$ and $B_{u}^+$. By the same argument, no arcs between $H(v_O)$ and $H(v)$ are used twice and thus $(B_{u}^+,B_{v}^-)$ is a good $(u,v)$-pair in $Q$.
\end{proof}

\begin{lem}\label{uneqvsameQ}
	Let $S$ be a semicomplete digraph on $s\geq 2$ vertices and let $H_1,\ldots,H_s$ be arbitrary digraphs. Let $Q=S[H_1,\ldots,H_s]$ and let $u,v$ be two distinct vertices of $Q$ such that $u$ and $v$ belong to the same $H_r$. Suppose that there exists a good $(u_S,v_S)$-pair in $S$.  If $d_Q^+(u)\geq 2$ and $d_Q^-(v)\geq 2$, then $Q$ has a good $(u,v)$-pair.
\end{lem}
\begin{proof} 
	 It follows by Remark \ref{Remark-uSvS} that $u_S=v_S=u$. Let $(O^S,I^S)$ be a good $(u,u)$-pair in $S$.  Suppose that $u_Iu\in A(I^S)$ and $uu_O\in A(O^S)$, which implies that every vertex of $H(u_I)$ (resp., $H(u)$) dominates every vertex of $H(u)$ (resp., $H(u_O)$) in $Q$. We first extend $(O^S,I^S)$ to a good pair $(O,I)$ rooted at $u$ in $Q-v$ as follows: $A(O)=\{xr:xy\in A(O^S), r\in H(y)\}\cup \{u_Ir:r\in H(u)-\{u,v\}\}$ and $A(I)=\{ry:xy\in A(I^S), r\in H(x)\}\cup \{ru_O:r\in H(u)-\{u,v\}\}$. 
	
	If $uv\in A(Q)$, then $(O\cup \{u_Iv\},I\cup \{uv\})$ is a good $(u,v)$-pair in $Q$. So it suffices to consider the case that $u$ does not dominate $v$.  Then $u$ has out-degree at least 2 and $v$ has in-degree at least 2 in $Q-\{u,v\}$. Let $u_1,\ldots,u_p$, $p\geq 2$ be out-neighbors of $u$ and let $v_1,\ldots,v_q$, $q\geq 2$ be in-neighbors of $v$ in $Q-\{u,v\}$, respectively. Since $u$ and $v$ have the same out- and in-neighbors in $Q-H(u)$ and no arc of $\cup_{i=1}^tA(H_i)$ is used in $O\cup I$, we can assume that $u_1,\ldots,u_p$ and $v_1,\ldots,v_q$ are labeled such that the out-arcs of $u$ in $O$ are $\{uu_i: i\in [p^{\prime}]\}$, where $1\leq p^{\prime}\leq p$ and, the arcs in $\{v_iu:i\in[q^{\prime}]\}$ with $1\leq q^{\prime}\leq q$ belong to $I$.
	
	Suppose first that the following condition holds possibly after permuting $u_1,\ldots, u_{p^{\prime}}$. 
	\begin{equation}
		\begin{aligned}\label{eq1}
			\mbox{The arc } v_1u \mbox{ belongs to the } (u_2,u)\mbox{-path in }I \mbox{ and, there exists } \\r\in[q], r>1 \mbox{ such that there is no } (u_2,v_r)\mbox{-path in }O\mbox{ when } uu_2\in A(O).
		\end{aligned}
	\end{equation}
\begin{figure}[H]
	\subfigure{\begin{minipage}[t]{0.49\linewidth}
			\centering\begin{tikzpicture}[scale=0.5]
				\draw[ line width=0.8pt] (0,0) ellipse [x radius=60pt, y radius=40pt];
				\filldraw[black](-1,0) circle (3pt)node[label=above:$u$](u){};
				\filldraw[black](1,0) circle (3pt)node[label=above:$v$](v){};
				\filldraw[black](3,-3) circle (3pt)node[label=right:$u_1$](u1){};
				\filldraw[black](-5,-3) circle (3pt)node[label=left:$u_2$](u2){};
				\filldraw[black](-3,-3) circle (3pt)node[label=below:$v_1$](v1){};
				\filldraw[black](1,-3) circle (3pt)node[label=below:$v_2$](v2){};
				\foreach \i/\j in {u2/v1,v1/u,v2/u,u1/v2}{\path[draw, blue, dashed, line width=0.8pt] (\i) edge (\j);}
				\foreach \i/\j in {u/u1,u/u2}{\path[draw, red, line width=1.8pt] (\i) edge (\j);}
				\path[draw, red, line width=1.8pt] (u2) edge[bend right=40] (v2);
				\path[draw, red, line width=1.8pt] (u1) edge[bend left=40] (v1);
			\end{tikzpicture}\caption*{(a)}\end{minipage}}~~~~\subfigure{\begin{minipage}[t]{0.49\linewidth}
			\centering\begin{tikzpicture}[scale=0.5]
				\draw[ line width=0.8pt] (0,0) ellipse [x radius=60pt, y radius=40pt];
				\filldraw[black](-1,0) circle (3pt)node[label=above:$u$](u){};
				\filldraw[black](1,0) circle (3pt)node[label=above:$v$](v){};
				\filldraw[black](3,-3) circle (3pt)node[label=right:$u_1$](u1){};
				\filldraw[black](-5,-3) circle (3pt)node[label=left:$u_2$](u2){};
				\filldraw[black](-3,-3) circle (3pt)node[label=below:$v_1$](v1){};
				\filldraw[black](1,-3) circle (3pt)node[label=below:$v_2$](v2){};
				\foreach \i/\j in {u2/v1,v1/u,v2/u,u1/v2}{\path[draw, blue, dashed, line width=0.8pt] (\i) edge (\j);}
				\foreach \i/\j in {u/u1,u/u2,v1/v2}{\path[draw, red, line width=1.8pt] (\i) edge (\j);}
				\path[draw, line width=0.8pt] (u2) edge[bend right=40] (v2);
				\path[draw, red, line width=1.8pt] (u1) edge[bend left=40] (v1);
			\end{tikzpicture}\caption*{(b)}\end{minipage}}
	\caption{Figures in Lemma \ref{uneqvsameQ}. The red fat (resp., blue dashed) arcs indicates the out- (resp., in-) branching with root $u$ in $Q-v$.}\label{fig-lemma43}
\end{figure}

	Then one can obtain a good $(u,v)$-pair in $Q$ as follows: $(I-v_1u)\cup\{v_1v,uu_2\}$ is the in-branching and construct the out-branching from $O\cup\{v_rv\}$ by adding the arc $vu_2$ if $uu_2\in A(O)$. Note that the condition (\ref{eq1}) guarantees that no cycle occurs in either branchings after the modification.
	
Thus it suffices to show that there exists a good pair rooted at $u$ in $Q-v$ which satisfies the condition (\ref{eq1}). By relabeling $v_1,\ldots,v_q$ if necessary, we may assume that the first statement of (\ref{eq1}) holds. If the second part of the condition does not hold, then $uu_2\in A(O)$ and $uu_2$ is on the $(u,v_r)$-path in $O$ for each $r>1$. As $O$ is an out-branching, there is no $(u_1,v_r)$-path in $O$ for any $r>1$. So if $v_1u$ also belongs to the $(u_1,u)$-path in $I$, then we can exchange the labels of $u_1$ and $u_2$ and then (\ref{eq1}) holds. Thus we may assume that $v_2u$ belongs to the $(u_1,u)$-path in $I$, which implies that $d_I^-(u)\geq 2$. Moreover,  there is a $(u_1,v_1)$-path in $O$, otherwise (\ref{eq1}) will hold with $r=2$ when we exchange the labels of $u_1$ and $u_2$ and the labels of $v_1$ and $v_2$.

It should be noted that $v_1\neq u_{2}$ and $v_2\neq u_{1}$ as there is a $(u_{i},v_{i})$-path in the out-branching $O$ for each $i\in[2]$. Since $v_1$ and $v_2$ are dominated by distinct vertices in $O$, we conclude that $v_1,v_2\notin H(u)$ and $H(v_1)\neq H(v_2)$ by the construction of $O$. Thus $H(v_i)$ dominates $H(v_{3-i})$ for some $i\in[2]$ as $S$ is semicomplete. We may w.l.o.g. assume that $i=1$ and then $v_1v_2\in A(Q)$ (see Figure \ref{fig-lemma43}). Then the in-branching $I$ and the out-branching obtained from $O$ by deleting the in-arc of $v_2$ in $O$ and adding the arc $v_1v_2$ form a good pair rooted at $u$ in $Q-v$ satisfying condition (\ref{eq1}) with $r=2$, which completes the proof. 
\end{proof}

\subsection{{\bf When $S$ has branchings $B^+_{u_S},B^-_{v_S}$ that  share only one arc}}

We start by recalling the characterization of semicomplete digraphs without good $(u,u)$-pairs.

\begin{thm}\cite{bangJGT20b} \label{sABC}
 	Let $S$ be a strong semicomplete digraph and let $u\in V(S)$ be arbitrary vertex. Suppose that $S$ does not contain a good $(u,u)$-pair. Then the following holds where $X, Y, Z$ form a partition of $V(S)-u$ such that $N^+_S(u) = X\cup Z$ and $N^-_S(u)= Y\cup Z$, where $Z$ is the set of vertices that form a 2-cycle with $u$:  There is precisely one arc $e$ leaving the terminal component of $S\left\langle X \right\rangle$ and precisely one arc $e^{\prime}$ entering the initial component of $S\left\langle Y\right\rangle$ and $e=e^{\prime}$. In particular $(S,u,u)$ is of type A with $\alpha=1$ and backward arc $e$.
      \end{thm}

      We first consider the case when $u$ and $v$ correspond to the same vertex of $S$.
      
\begin{lem}\label{uneqvsameQ1}
	Let $S$ be a strong semicomplete digraph on $s\geq 2$ vertices and let $H_1,\ldots,H_s$ be arbitrary digraphs. Let $Q=S[H_1,\ldots,H_s]$ and let $u,v$ be two vertices of the same $H_r$ for some $r\in [s]$ with $d_Q^+(u)\geq2, d_Q^-(v)\geq2$ if $u\neq v$. Suppose that $Q$ is not 2-arc-strong and that $S$ has  no good $(u_S,v_S)$-pair. Then $(S,u_S,v_S)$ is of type A with $\alpha=1$ and backward arc $xy$ and if $Q$ has no good $(u,v)$-pair,then one of the following statements holds.
	
	(i)  $|H(x)|=|H(y)|=1$.
	
	(ii) $|H(x)|\geq 2$ and $d_Q^+(w)=1$ for every $w\in H(x)$, or $|H(y)|\geq 2$ and $d_Q^-(w)=1$ for every $w\in H(y)$.
\end{lem}
\begin{proof} 
  It follows by Remark \ref{Remark-uSvS} that $u_S=v_S=u$. Since $S$ has no good $(u,u)$-pair, by Theorem \ref{sABC}, there is a partition $X, Y, Z$ of $V(S)-u$ such that $N^+_S(u) = X\cup Z$ and $N^-_S(u)= Y\cup Z$, where $Z$ is the set of vertices that form a 2-cycle with $u$. Further, $(S,u_S,v_S)$ is of type A with $\alpha=1$ and backward arc $xy$, where $xy$ is the only arc leaving the terminal component of $S\left\langle X \right\rangle$ (resp., entering the initial component of $S\left\langle Y\right\rangle$).  As $x$ belongs to the terminal component of $S\left\langle X \right\rangle$
there is an in-branching $B_{x,S\left\langle X \right\rangle}^-$ rooted at $x$ in  $S\left\langle X \right\rangle$. Similarly, as $y$ belongs to the initial component of $S\left\langle Y \right\rangle$, there is an out-branching $B_{y,S\left\langle Y \right\rangle}^+$ rooted at $y$ in  $S\left\langle Y \right\rangle$. Using these branchings, one easily obtains a good $(u,u)$-pair $B_{u,S}^+, B_{u,S}^-$ in $S$ such that $A(B_{u,S}^+)\cap A(B_{u,S}^-)=\{xy\}$.

Suppose that conditions (i) and (ii) do not hold.  We may assume w.l.o.g.  that $|H(x)|\geq 2$ and by Remark \ref{Remark-copyofD} we may choose $\{x,x^{\prime}\}\subseteq H(x)$ such that $d_Q^+(x^{\prime})\geq2$ and, if $x^{\prime}$ has an out-neighbor in $H(x)$, then assume that $x$ is such an neighbor, i.e., $x^{\prime}x\in A(Q)$. 
 Let $e$ be an out-arc of $x^{\prime}$ in $Q$ which is distinct from  $x^{\prime}y$. Since $xy$ is the only one arc in $S$ leaving the terminal component of $S\left\langle X \right\rangle$ and $x^{\prime}\in H(x)$, the arc $e$ goes to some vertex in the terminal component of
 $S\left\langle X \right\rangle$.

If $u=v$, then
  $(B_{u,S}^+-xy)\cup\{ux^{\prime},x^{\prime}y\}$ and $B_{u,S}^-\cup\{e\}$ form a good $(u,u)$-pair in $Q\left\langle V(S)\cup\{x^{\prime}\} \right\rangle$ and there is a good $(u,u)$-pair in $Q$ by Lemma \ref{(D-X)->D}. So $u\neq v$ and by our assumption we have $d_Q^-(v)\geq2$.  Let $y^{\prime}$ be a copy of $y$ if $|H(y)|\geq 2$ and otherwise $y=y^{\prime}$.

Let $Q^{\prime}=Q\left\langle X\cup Y\cup \{u,v,x^{\prime}\}\cup\{y^{\prime}\}\right\rangle$.  Let $$O=B_{y,S\left\langle Y \right\rangle}^+\cup\{ur:r\in  (X-x)\cup\{x^{\prime}\}\}\cup\{vx,x^{\prime}y\}\mbox{ and}$$ $$I=B_{x,S\left\langle X \right\rangle}^-\cup \{rx:r\in (Y-y)\cup\{u\}\}\cup\{xy,yv\}.$$

    
First we consider the case that $|H(y)|=1$, that is, $y=y^{\prime}$. Recall that $v$ has in-degree two and $e$ is an out-arc of $x^{\prime}$ in $Q$ which is distinct from  $x^{\prime}y$. 
Let $z$ be an in-neighbor of $v$ in $Q-y$. Note that  that $z$ belongs to some $H(r)$ with
$r\in Y\cup Z\cup \{u\}$. If $z\in V(Q^{\prime})$, that is, $z$ belongs to $(Y-y)\cup\{u\}$, then $(O\cup\{zv\},I\cup\{e\})$ is a good $(u,v)$-pair in $Q^{\prime}$. See Figure \ref{fig-lemma55} (a). Otherwise, $(O\cup\{yz,zv\},I\cup\{e,zx\})$ is a good $(u,v)$-pair in $Q\left\langle V(Q^{\prime})\cup\{z\}\right\rangle$.  In both cases, $Q$ has a good $(u,v)$-pair by Lemma \ref{(D-X)->D}.
	\begin{figure}[H]
		\subfigure{\begin{minipage}[t]{0.48\linewidth}
				\centering\begin{tikzpicture}[scale=0.35]
					\foreach \i in {(5,0),(0,0),(-5,0)}{\draw[ line width=0.8pt] \i ellipse [x radius=50pt, y radius=80pt];}
					\coordinate [label=center:$Y$] () at (-5,4);
					\coordinate [] () at (-0,4);
					\coordinate [label=center:$X\cup\{x^{\prime}\}$] () at (5,4);
					\coordinate [label=center:$e$] () at (6.3,-1);
					\draw[-stealth,line width=1.8pt] (-2,3.2) -- (2,3.2); 
					\filldraw[black](0,1) circle (3pt)node[label=above:$u$](u){};
					\filldraw[black](0,-1.5) circle (3pt)node[label=below:$v$](v){};
					\filldraw[black](-5,1) circle (3pt)node[label=above:$z$](z){};
					\filldraw[black](-5,-1.5) circle (3pt)node[label=below:$y$](y){};
					\filldraw[white](-5,-4) circle (3pt)node(y1){};
					\filldraw[black](5,1) circle (3pt)node(x0){};
					\filldraw[black](5,-1.5) circle (3pt)node[label=below:$x$](x){};
					\filldraw[black](5,-4) circle (3pt)node[label=right:$x^{\prime}$](x1){};
					\foreach \i/\j/\c/\t/\a in {
						z/x/blue/0/0.8,
						u/x/blue/0/0.8,
						x0/x/blue/0/0.8,
						y/v/blue/0/0.8,
						x/y/blue/30/0.8,
						x1/x0/blue/-30/0.8
					}{\path[draw,dashed, \c, line width=\a] (\i) edge[bend left=\t] (\j);}
				\foreach \i/\j/\c/\t/\a in {
					y/z/red/0/1.5,
					z/v/red/0/1.5,
					u/x0/red/0/1.5,
					u/x1/red/0/1.5,
					v/x/red/0/1.5,
					x1/y/red/30/1.5,
					x1/y1/white/30/1.5
				}{\path[draw, \c, line width=\a] (\i) edge[bend left=\t] (\j);}
				\end{tikzpicture}\caption*{(a) $|H(y)|=1$}\end{minipage}}
		\subfigure{\begin{minipage}[t]{0.48\linewidth}
				\centering\begin{tikzpicture}[scale=0.35]
					\foreach \i in {(5,0),(0,0),(-5,0)}{\draw[ line width=0.8pt] \i ellipse [x radius=50pt, y radius=80pt];}
					\coordinate [label=center:$Y$] () at (-5,4);
					\coordinate [] () at (-0,4);
					\coordinate [label=center:$X\cup\{x^{\prime}\}$] () at (5,4);
					\draw[-stealth,line width=1.8pt] (-2,3.2) -- (2,3.2); 
					\filldraw[black](0,1) circle (3pt)node[label=above:$u$](u){};
					\filldraw[black](0,-1.5) circle (3pt)node[label=below:$v$](v){};
					\filldraw[black](-5,1) circle (3pt)node[label=above:$z$](z){};
					\filldraw[black](-5,-1.5) circle (3pt)node[label=below:$y$](y){};
					\filldraw[black](-5,-4) circle (3pt)node[label=left:$y^{\prime}$](y1){};
					\filldraw[black](5,1) circle (3pt)node(x0){};
					\filldraw[black](5,-1.5) circle (3pt)node[label=below:$x$](x){};
					\filldraw[black](5,-4) circle (3pt)node[label=right:$x^{\prime}$](x1){};
					\foreach \i/\j/\c/\t/\a in {
						z/x/blue/0/0.8,
						x0/x/blue/0/0.8,
						u/x1/blue/0/0.8,
						y/v/blue/0/0.8,
						x/y/blue/30/0.8,
						x1/y1/blue/30/0.8,
						y1/z/blue/30/0.8
					}{\path[draw, dashed,\c, line width=\a] (\i) edge[bend left=\t] (\j);}
				\foreach \i/\j/\c/\t/\a in {
					y/z/red/0/1.5,
					u/x0/red/0/1.5,
					u/x/red/0/1.5,
					y1/v/red/0/1.5,
					v/x1/red/0/1.5,
					x/y1/red/30/1.5,
					x1/y/red/30/1.5
				}{\path[draw, \c, line width=\a] (\i) edge[bend left=\t] (\j);}
				\end{tikzpicture}\caption*{(b) $|H(y)|\geq 2$}\end{minipage}}
			\caption{Figures in Lemma \ref{uneqvsameQ1}. The red fat arcs and the blue dashed arcs indicate the out- and in-branchings, respectively.}\label{fig-lemma55}
	\end{figure}
	So it suffices to consider the case that $y$ and $y^{\prime}$ are two distinct vertices. First we adjust $O$ by letting $O^{\prime}=(O-\{vx,ux^{\prime}\})\cup \{vx^{\prime},ux,xy^{\prime},y^{\prime}v\}$ and adjust $I$ by letting $I^{\prime}=(I-ux)\cup \{ux^{\prime},x^{\prime}y^{\prime}\}$. If $|Y|\geq 2$ or $|X|\geq 2$, then pick one out-neighbor $z$ of $y$ (consequently, the out-neighbor of $y^{\prime}$) in $(X-x)\cup(Y-y)$. Such neighbor exists as there is only one arc $xy$ entering initial component of $S-xy$. Then $(O^{\prime},I^{\prime}\cup\{y^{\prime}z\})$ is a good $(u,v)$-pair in the digraph $Q^{\prime}$. See Figure \ref{fig-lemma55} (b). For the case that $|Z|\geq 1$, let $z\in Z$. Observe that $y$ dominates all vertices in $Z$ and each vertex of $Z$ dominates $x$. Then $(O^{\prime}\cup\{yz\},I^{\prime}\cup\{y^{\prime}z,zx\})$ is a good $(u,v)$-pair in the digraph $Q\left\langle V(Q^{\prime})\cup\{z\}\right\rangle$. In both cases, $Q$ has a good $(u,v)$-pair by Lemma \ref{(D-X)->D}. So we may assume that $Z=\emptyset$, $X=\{x\}$ and $Y=\{y\}$, which implies that $S=uxyu$. Since each of $H(u),H(x)$ and $H(y)$ has order at least 2 (recall that $u$ and $v$ are two distinct vertices in $H(u)$), the digraph $Q$ is 2-arc-strong, which contradicts our assumption.
      \end{proof}

  It remains to consider the case when $u$ and $v$ correspond to distinct vertices of $S$ and hence $H(u)\neq H(v)$. 

\begin{lem}\label{shareonearc}
  Let $S$ be a strong semicomplete digraph on $s\geq 2$ vertices and let $H_1,\ldots,H_s$ be arbitrary digraphs. Let $Q=S[H_1,\ldots,H_s]$ and let $u,v$ be two vertices of $Q$ with $H(u)\neq H(v)$ and $d_Q^+(u)\geq 2$. Suppose that $S$ has no good $(u_S,v_S)$-pair but that it has branchings $B_{u_S,S}^+,B_{v_S,S}^-$ such that $A(B_{u_S,S}^+)\cap A(B_{v_S,S}^-)=\{xy\}$. Suppose that there is no good $(u,v)$-pair in $Q$. If $|H(x)|\geq 2$ and there exists a vertex $w$ in $H(x)$ with $d_{Q}^+(w)\geq 2$, then the following holds:
  \begin{itemize}
  \item[(a)] $x=u,y=v$.
  \item[(b)]  Each vertex in $H(u)-u$ (resp., $H(v)-v$) has in-degree (resp., out-degree) exactly one in $Q$.
    \item[(c)] $Q$ is isomorphic to one of the digraphs shown in Table \ref{exceptions} (d)-(f).
  \end{itemize}
\end{lem}
\begin{proof}
	 It follows by Remark \ref{Remark-uSvS} that $u_S=u$ and $v_S=v$. So the vertices $u,v,x$ and $y$ all  belong to $V(S)$. Analogous to the construction in the proof of Lemma \ref{goodinD}, for a given pair $(B_{u,S}^+,B_{v,S}^-)$ of $S$ with $A(B_{u,S}^+)\cap A(B_{v,S}^-)=\{xy\}$ and any in-neighbor $u_I$ of $u$ in $Q-H(u)$, one can obtain branchings $B_{u,Q}^+$ and $B_{v,Q}^-$ in $Q$ such that $A(B_{u,Q}^+)\cap A(B_{v,Q}^-)=\{xy\}$ as follows: construct $B_{u,Q}^+$ from $B_{u,S}^+$ by adding arcs $\{ar:ab\in A(B_{u,S}^+),r\in H(b)\}$ and $\{u_Ir:r\in H(u)-u\}$ and construct $B_{v,Q}^-$ from $B_{v,S}^-$ by adding arcs $\{rb:ab\in A(B_{v,S}^-),r\in H(a)\}$ and $\{rv_O:r\in H(v)-v\}$, where $v_O$ is an out-neighbor of $v$ in $S$. The vertices $u_I$ and $v_O$ exist as $S$ is strong. Note that  $\{rv_O:r\in H(v)-v\}$ and $\{u_Ir:r\in H(u)-u\}$ share no common arc as $v_O=u$ if $v_O\in H(u)$. Further, by construction, none of $B_{u,Q}^+$ and $B_{v,Q}^-$ contain an arc in $\cup_{i\in[t]}A(H_i)$.  
	
         Assume that $\{x_1,x_2\}\subseteq H(x)$ with $x_1=x=w$ and $d_{Q}^+(x_1)\geq 2$ by Remark \ref{Remark-copyofD}. It should be noted that $u=x=x_1$ if $u\in H(x)$ as $u,x\in V(S)$. By the construction of $B_{u,Q}^+$ and $B_{v,Q}^-$, the arc $x_1y$ belongs to $A(B_{u,Q}^+)$, and $x_1y$ and $x_2y$ belong to $A(B_{v,Q}^-)$.  If $x_1$ has an out-neighbor in $H(x)$, w.l.o.g say
         $x_1x_2\in H(x)$, then $B_{u,Q}^+$ and $(B_{v,Q}^--x_1y)\cup \{x_1x_2\}$ form a good $(u,v)$-pair in $Q$, contradicting our assumption.  So  $x_1$ has out-degree zero in $H(x)$ and then there is an out-neighbor $y^{\prime}$ of $x_1$ in $Q-H(x)-y$. Note that as the arcs $x_1y, x_2y$ belong to the in-branching $B_{v,Q}^-$, we have  $x_1y^{\prime},x_2y^{\prime}\notin A(B_{v,Q}^-)$. Moreover, by the way we constructed $B_{u,Q}^+$ and the fact that $x_1=x$ belongs to $V(S)$, we have $x_2y^{\prime}\notin B_{u,Q}^+$ and $x_1y^{\prime}$ may belong to $B_{u,Q}^+$. Construct $O$ from $B_{u,Q}^+$ by deleting the in-arcs of $y,y^{\prime}$ in $B_{u,Q}^+$ and adding arcs  $x_1y^{\prime},x_2y$ and let $I=(B_{v,Q}^--x_2y)\cup\{x_2y^{\prime}\}$. Let $(O^{\prime},I^{\prime})$ be a pair obtained from $(O,I)$ by exchanging the four edges between $\{x_1,x_2\}$ and $\{y,y^{\prime}\}$, that is, $O^{\prime}=(O-\{x_1y^{\prime},x_2y\}\cup\{x_1y,x_2y^{\prime}\})$ and  $I^{\prime}=(I-\{x_1y,x_2y^{\prime}\}\cup\{x_1y^{\prime},x_2y\})$ (see Figure \ref{fig-lemma44}).

\begin{figure}[H]
	\subfigure{\begin{minipage}[t]{0.33\linewidth}
			\centering\begin{tikzpicture}[scale=0.5]
				\draw[ line width=0.8pt] (0,-0.5) ellipse [x radius=30pt, y radius=60pt];
				\coordinate [label=center:$H(x)$] () at (0,2.5);
				\filldraw[black](0,1) circle (3pt)node[label=below:$x_1$](x1){};
				\filldraw[black](0,-1) circle (3pt)node[label=below:$x_2$](x2){};
				\filldraw[black](2,1) circle (3pt)node[label=below:$y$](y1){};
				\filldraw[black](2,-1) circle (3pt)node[label=below:$y^\prime$](y2){};
				\filldraw[white](-2,1) circle (0.1pt)node(a1){};
				\filldraw[white](-2,0) circle (0.1pt)node(a2){};
				\filldraw[white](4,1) circle (0.1pt)node(b1){};
				\filldraw[white](4,0) circle (0.1pt)node(b2){};
				\foreach \i/\j in {a1/x1,a2/x1,x1/y1,x2/y1}{\path[draw, blue, dashed, line width=0.8pt] (\i) edge (\j);}
				\foreach \i/\j in {x1/y2,x2/y2}{\path[draw, line width=0.8pt] (\i) edge (\j);}
				\foreach \i/\j in {y1/b1,y1/b2}{\path[draw, red, line width=1.5pt] (\i) edge (\j);}
				\path[draw, red, line width=1.5pt] (x1) edge[bend left=30] (y1);
			\end{tikzpicture}\caption*{(a) $(B^+_{u,Q},B^-_{v,Q})$}\end{minipage}}
	\subfigure{\begin{minipage}[t]{0.33\linewidth}
			\centering\begin{tikzpicture}[scale=0.5]
				\draw[ line width=0.8pt] (0,-0.5) ellipse [x radius=30pt, y radius=60pt];
				\coordinate [label=center:$H(x)$] () at (0,2.5);
				\filldraw[black](0,1) circle (3pt)node[label=below:$x_1$](x1){};
				\filldraw[black](0,-1) circle (3pt)node[label=below:$x_2$](x2){};
				\filldraw[black](2,1) circle (3pt)node[label=below:$y$](y1){};
				\filldraw[black](2,-1) circle (3pt)node[label=below:$y^\prime$](y2){};
				\filldraw[white](-2,1) circle (0.1pt)node(a1){};
				\filldraw[white](-2,0) circle (0.1pt)node(a2){};
				\filldraw[white](4,1) circle (0.1pt)node(b1){};
				\filldraw[white](4,0) circle (0.1pt)node(b2){};
				\filldraw[white](3,0) circle (0.1pt)node(c){};
				\foreach \i/\j in {a1/x1,a2/x1,x1/y1,x2/y2}{\path[draw, blue, dashed, line width=0.8pt] (\i) edge (\j);}
				\foreach \i/\j in {x1/y2,x2/y1,y1/b1,y1/b2}{\path[draw, red, line width=1.5pt] (\i) edge (\j);}
				\path[draw, red, dotted, line width=1.5pt] (c) edge (y2);
			\end{tikzpicture}\caption*{(b) $(O,I)$}\end{minipage}}
		\subfigure{\begin{minipage}[t]{0.33\linewidth}
			\centering\begin{tikzpicture}[scale=0.5]
				\draw[ line width=0.8pt] (0,-0.5) ellipse [x radius=30pt, y radius=60pt];
				\coordinate [label=center:$H(x)$] () at (0,2.5);
				\filldraw[black](0,1) circle (3pt)node[label=below:$x_1$](x1){};
				\filldraw[black](0,-1) circle (3pt)node[label=below:$x_2$](x2){};
				\filldraw[black](2,1) circle (3pt)node[label=below:$y$](y1){};
				\filldraw[black](2,-1) circle (3pt)node[label=below:$y^\prime$](y2){};
				\filldraw[white](-2,1) circle (0.1pt)node(a1){};
				\filldraw[white](-2,0) circle (0.1pt)node(a2){};
				\filldraw[white](4,1) circle (0.1pt)node(b1){};
				\filldraw[white](4,0) circle (0.1pt)node(b2){};
				\filldraw[white](3,0) circle (0.1pt)node(c){};
				\foreach \i/\j in {a1/x1,a2/x1,x2/y1,x1/y2}{\path[draw, blue, dashed, line width=0.8pt] (\i) edge (\j);}
				\foreach \i/\j in {}{\path[draw, line width=0.8pt] (\i) edge (\j);}
				\foreach \i/\j in {x1/y1,x2/y2,y1/b1,y1/b2}{\path[draw, red, line width=1.5pt] (\i) edge (\j);}
				\path[draw, red, dotted, line width=1.5pt] (c) edge (y2);
			\end{tikzpicture}\caption*{(c) $(O^\prime,I^\prime)$}\end{minipage}}
		\caption{The red fat arcs are in  $B_{u,Q}^+, O$ and $O^{\prime}$, respectively, and the blue dashed arcs are in $B_{v,Q}^-, I$ and $I^{\prime}$, respectively. The red dotted arrow corresponds to the in-arc of $y^{\prime}$ in $B_{u,Q}^+$, possibly the in-arc is $x_1y^{\prime}$.}\label{fig-lemma44}
\end{figure}

	Next we show that for each out-neighbor $y^{\prime}$ of $x_1$ in $Q-H(x)-y$, the following statements hold.
	
\textbf{(A)} There is no $(y^{\prime},x_2)$-path in $B_{u,Q}^+$.

\textbf{(B)} There is a $(y^{\prime},x_1)$-path in $B_{v,Q}^-$.

	Suppose to the contrary that for some out-neighbor $y^{\prime}$ of $x_1$ there is a $(y^{\prime},x_2)$-path $P_{y^{\prime},x_2}$ in $B_{u,Q}^+$ (and consequently also in $O$ and $O^{\prime}$). By the construction of $B_{u,Q}^+$, if $u\notin H(x)$, then all vertices in $H(x)$ are dominated by the same vertex in $B_{u,Q}^+$ and if $u\in H(x)$, then $u=x_1$ by our assumption.  In both cases $x_1$ and $y$ are not on $P_{y^{\prime},x_2}$ as the path belongs to the out-branching $B_{u,Q}^+$. Then $B_{u,Q}^+$ does not contain a path from $y$ to any vertex in $P_{y^{\prime},x_2}-y^{\prime}$  as it is an out-branching.  Observe that there is no $(y,x_1)$-path in the in-branching $B_{v,Q}^-$ as $x_1y$ is an arc of $B_{v,Q}^-$ and by the construction of $B_{v,Q}^-$, the vertex $x_2$ has in-degree zero in $B_{v,Q}^-$, thus $I$ is an in-branching rooted at $v$ in $Q$.  There is clearly no $(y,x_1)$-path in the out-branching $B_{u,Q}^+$ as $x_1y\in B_{u,Q}^+$. Combining this and the fact that $x_1$ is the only in-neighbor of $y^{\prime}$ in $O$, there is no $(y,y^{\prime})$-path in $O$ and then $x_2y$ belongs to no cycle in $O$. By our assumption, $(O,I)$ is not a good $(u,v)$-pair in $Q$ and hence $O$ must contain a cycle and by construction, this cycle contains the arc $x_1y^{\prime}$, that is, there is a $(y^{\prime},x_1)$-path in $O$ (and hence in $B_{u,Q}^+$). Therefore, $x_1\neq u$ and $x_1y^{\prime}\notin B_{u,Q}^+$. 
	
	By the construction of $B_{u,Q}^+$ and $O$, we have that the predecessor $x_2^-$ of $x_2$ on the path $P_{y^{\prime},x_2}$ dominates $x_1$ in $B_{u,Q}^+$ (and consequently in $O$). If there is no $(y^{\prime},x_1)$-path in $B_{v,Q}^-$, then $(B_{u,Q}^+,(B_{v,Q}^--x_1y)\cup\{x_1y^{\prime}\})$ is a good $(u,v)$-pair in $Q$, which contradicts our assumption. So it suffices to consider the case that there is a $(y^{\prime},x_1)$-path $P_{y^{\prime},x_1}$ in $B_{v,Q}^-$. Let $x_1^-$ be the predecessor of $x_1$ on the path. Clearly $x_1^-\neq x_2^-$ as $x_2^-x_1\in B_{u,Q}^+$ and $x_1^-x_1\in B_{v,Q}^-$. Hence $x_1^-x_2\notin B_{u,Q}^+$ as $x_2^-x_2\in B_{u,Q}^+$. Then $B_{u,Q}^+$ and  $(B_{v,Q}^--x_1y-x_1^-x_1)\cup\{x_1y^{\prime},x_1^-x_2\}$ form a good $(u,v)$-pair in $Q$, a contradiction again. Therefore, (A) holds.

 Suppose that there is an out-neighbor $y^{\prime}$ such that there is no $(y^{\prime},x_1)$-path. It follows by (A) that there is no $(y^{\prime},x_2)$-path in $B_{u,Q}^+$ and then  $(O^{\prime},I^{\prime})$ is a good $(u,v)$-pair in $Q$. This contradicts our assumption. So (B) holds.

For a given out-neighbor $y^{\prime}$ of $x_1$ in $Q-H(x)-y$, let $z$ be the predecessor of $x_1$ on the path $P_{y^{\prime},x_1}$ in $B_{v,Q}^-$. Here it should be noted that there is no $(y,z)$-path in $B_{v,Q}^-$ otherwise $x_1y$ would belong to a cycle of the in-branching $B_{v,Q}^-$. So if $zx_2\notin B_{u,Q}^+$, then $O^{\prime}$ and $(I^{\prime}-zx_1)\cup \{zx_2\}$ form a good $(u,v)$-pair in $Q$, a contradiction again. Thus $zx_2\in B_{u,Q}^+$. By the construction of $B_{u,Q}^+,B_{v,Q}^-$ and the fact that $zx_1\in B_{v,Q}^-,zx_2\in B_{u,Q}^+$, we have that $u=x_1=x=u_S$ (in $V(S)$) and $z=u_I$. Further, by the arbitrariness of the in-neighbor $u_I$ of $u$, we may assume that each vertex in $H(x)$, i.e., $H(u)$, has exactly one in-neighbor $z$ in $Q-H(u)$.  Then (B) implies that $y$ is the only possible vertex which $v$ can be as every other vertex has out-degree one in $B_{v,Q}^-$.  This establishes (a) in the statement of the lemma.

If there exists $x_i\in H(u)-u$ such that $x_i$ has in-degree at least two in $Q$, we can assume w.l.o.g that $x_i=x_2$ and $x_3\in H(u)$ is an in-neighbor  of $x_2$. By the previously established fact that $x_1$ has no out-neighbor in $H(x)=H(u)$, we must have $x_3\neq x_1$. Then $(O^{\prime}-zx_2)\cup\{x_3x_2\}$ and $(I^{\prime}-zx_1)\cup \{zx_2\}$ form a good $(u,v)$-pair in $Q$, a contradiction. Thus each vertex in $H(u)-u$ has exactly one in-neighbor, namely $z$ in $Q$. Since for each out-neighbor $y^{\prime}$ of $u$ in $Q-y$, there is no $(y^{\prime},x_2)$-path in $B_{u,Q}^+$ by (A), we have that there is a $(y,z)$-path in $B_{u,Q}^+$. More precisely, $yz\in A(B_{u,Q}^+)$. Otherwise, there is an internal vertex $w$ in the $(y,z)$-path and then $w$ has a path to $x_2$ and $u$ dominates $w$, which contradicts (A).

Suppose that there exists a vertex $v^{\prime}\in H(v)-v$ with out-degree at least two in $Q$. Then $u$ dominates $v^{\prime}$ and $v^{\prime}$ dominates $z$. Let $z^{\prime}$ be another out-neighbor of $v^{\prime}$ in $Q$ which is distinct from$z$. Choose $v^{\prime}$ as the vertex $y^{\prime}$, construct an in-branching rooted at $v$ from $I$ by deleting the out-arc of $y^{\prime}$ in $I$ and adding the arc $y^{\prime}z^{\prime}$. Since $x_2y^{\prime}\in A(I)$ and $x_2$ has in-degree zero in $I$, the arc $y^{\prime}z^{\prime}$ does not belong to a cycle in the new resulting in-branching. Again, by the construction of $O$ and the fact $x_1y\in A(B_{u,Q}^+)$, there is no $(y,x_1)$-path in $O$. So the in-branching and $(O-vz)\cup\{y^{\prime}z\}$ form a good $(u,v)$-pair in $Q$, a contradiction again. Thus each vertex in $H(v)-v$ out-degree exactly one in $Q$, implying that (b) in the statement of the lemma holds.

Next we show the possible structures of $Q$. Recall that $z$ is the only in-neighbor of each vertex of $H(u)-u$, consequently, is the only in-neighbor of $u$ in $S$ as $S$ is strong. Since $z$ is the predecessor of $x_1$ on the path $P_{y^{\prime},x_1}$ in the  branching $B_{v,Q}^-$, we get that $z\neq v$ and then $Q$ is not isomorphic to one of the digraphs in Figure \ref{fig-SDexceptions} (c)-(f). Further, by the definition of a composition, $|H(z)|=1$. Since $S$ has no good $(u,v)$-pair but it has branchings $B_{u,S}^+, B_{v,S}^-$ such that $A(B_{u,S}^+)\cap A(B_{v,S}^-)=\{xy\}=\{uv\}$, Lemma \ref{noGP} shows that $(S,u,v)$ is of type B with partition $V_1,V_2$ ($\beta=1$) and backward arc set $\{uv\}$. For the case that $z\in V_2$, we clearly have $V_1=\{v\}$. We next show that if $z\in V_1$, then $V_2=\{u\}$. Otherwise, as $z$ is the only in-neighbor of $u$ in $S$, we have that there is no arc from $V_2-u$ to $V_1\cup\{u\}$ in $S$, contradicting the fact that $S$ is strong. Recall that $z\neq v$ and $|H(z)|=1$. So if $z\in V_1$, then $Q=C_3[H(u),H(v),z]$ (the digraph in Table \ref{exceptions} (d)).  For the case that $z\in V_2$ and $|H(v)|\geq2$, since $H(v)-v$ has exactly one out-neighbor in $Q$, we also have $Q=C_3[H(u),H(v),z]$.  Here it should be noted that in both cases $S$ is a 3-cycle, that is, $z$ does not dominate $v$  and $u$ does not dominate $z$ in $S$. Otherwise, it is not difficult to check that $Q$ has a good $(u,v)$-pair by the assumption $d_Q^+(u)\geq 2$ and $|H(u)|\geq 2$, a contradiction.

It suffices to consider the case that $z\in V_2$, $|V_1|=|H(v)|=1$ and $V(Q-H(u))-\{v,z\}$ is not empty. In this case, as $z$ is the only in-neighbor of $u$ in $S$ and $A(B_{u,S}^+)\cap A(B_{v,S}^-)=\{uv\}$, there is an in-branching rooted at $z$ in $Q-H(u)-v$, say $I_z$. If there is a $(u,z)$-path $P$ in $Q-v$ which is arc-disjoint with $I_z$, then one can construct a good $(u,v)$-pair in $Q$ as follows, which contradicts our assumption: Let $u_O$ be the successor of $u$ on $P$ and  let $I_z\cup \{zu,uv\}\cup\{ru_O:r\in H(u)-u\}$ be the in-branching and construct the out-branching from the path $P$ by adding arcs $\{zr:r\in H(u)-u\}\cup\{vr:r\in Q-H(u)-P\}$ and an arc from $H(u)-u$ to $v$. 

So we may assume that there is no such path $P$, which implies that $u$ does not dominate $z$ and, for any $w\in Q-H(u)-\{z,v\}$ there is no in-branching rooted at $z$ which is arc-disjoint from some $(w,z)$-path in $Q-H(u)-v$. In particular, no such branching and path in $S-\{u,v\}$. It follows by (the symmetrical form of) Lemma \ref{OutbranPath} that $z$ has exactly one in-neighbor $z^-$ in $S-\{u,v\}$. Further, if $|H(z^-)|\geq 2$, then $z$ is the only out-neighbor of $H(z^-)$ as there is no such branching and path in $Q-H(u)-v$. This implies that $Q$ is isomorphic to one of the digraphs shown in Table \ref{exceptions}  (e)-(f), where $z^-$ is the vertex in $S$ which corresponds to $\overline{K}_t$. 
\end{proof}

By symmetry, we have the following corollary.
\begin{coro}\label{shareonearcsymmetry}
	Let $S$ be a strong semicomplete digraph on $s\geq 2$ vertices and let $H_1,\ldots,H_s$ be arbitrary digraphs. Let $Q=S[H_1,\ldots,H_s]$ and let $u,v$ be two vertices of $Q$ with $H(u)\neq H(v)$ and $d_Q^-(v)\geq 2$. Suppose that $S$ has no good $(u_S,v_S)$-pair but that it has branchings $B_{u_S,S}^+,B_{v_S,S}^-$ such that $A(B_{u_S,S}^+)\cap A(B_{v_S,S}^-)=\{xy\}$. Suppose that $Q$ has no good $(u,v)$-pair. If $|H(y)|\geq 2$ and there exists a vertex $w$ in $H(y)$ with $d_{Q}^-(w)\geq 2$, then $\overleftarrow{Q}$ is isomorphic to one of the digraphs shown in Table \ref{exceptions} (d)-(f), where $\overleftarrow{Q}$ is obtained from $Q$ be reversing all arcs and interchanging the names of $u$ and $v$.
      \end{coro}

      \subsection{{\bf Proof of  Theorem \ref{mainthm}}}

      For convenience we repeat the statement of the theorem.\\
        
        \noindent {\bf Theorem \ref{mainthm}} Let $S$ be a strong semicomplete digraph of order $s\geq 2$ and let $H_1,\ldots,H_s$ be arbitrary digraphs. Suppose that $Q=S[H_1,\ldots,H_s]$ and $u, v$ are two arbitrary vertices of $Q$ such that $d_{Q}^+(u)\geq 2$ and $d_{Q}^-(v)\geq 2$ if $u\neq v$. Then $Q$ has a good $(u,v)$-pair if and only if it satisfies none of the following conditions.
	
	(i) $Q$ or $\overleftarrow{Q}$ is isomorphic to one of the digraphs in Table \ref{exceptions}, where $\overleftarrow{Q}$ is obtained from $Q$ be reversing all arcs and interchanging the names of $u$ and $v$. 	
	
	(ii) $(S,u_S,v_S)$ is of type A and for each backward arc $xy$, either $|H(x)|=|H(y)|=1$ or, $d_Q^+(w)=1$ for every $w\in H(x)$ if $|H(x)|\geq 2$ and $d_Q^-(w)=1$ for every $w\in H(y)$ if $|H(y)|\geq 2$.
	
	(iii) $(S,u_S,v_S)$ is of type B and there exists a backward arc $xy$ such that either $|H(x)|=|H(y)|=1$ or, $d_Q^+(w)=1$ for every $w\in H(x)$ if $|H(x)|\geq 2$ and $d_Q^-(w)=1$ for every $w\in H(y)$ if $|H(y)|\geq 2$.

\begin{proof}
  First we show that there is no good $(u,v)$-pair if one of (i)-(iii) holds. If $Q$ or $\overleftarrow{Q}$ is isomorphic to one of the digraphs described in Table \ref{exceptions} then it follows from Lemma \ref{specialexceptions} that $Q$ has no good $(u,v)$-pair. Recall that if $(S,u_S,v_S)$ is of type A (resp., type B), then any pair of branchings $B_{u_S,S}^+$ and $B_{v_S,S}^-$ of $S$ must share at least one backward arc (resp., all backward arcs) of $S$. This implies that for every good $(u,v)$-pair $(B_{u,Q}^+,B_{v,Q}^-)$ in $Q$, both $B_{u,Q}^+$ and $B_{v,Q}^-$ must use an arc from $H(x)$ to $H(y)$ for some backward arc $xy$ if $(S,a,b)$ is of type A (resp., for every backward arc $xy$ if $(S,a,b)$ is of type B). If  $|H(x)|=|H(y)|=1$ in $Q$, then there is clearly no good $(u,v)$-pair in $Q$. By symmetry, we may assume that $|H(x)|\geq 2$ and that we have $d_Q^+(w)=1$ for every $w\in H(x)$. Suppose that the arc $wy$ from $H(x)$ to $y$ is used in the out-branching $B_{u,Q}^+$. Then $w$ has out-degree zero in $Q-B_{u,Q}^+$ and thus it can not be collected into any in-branching with root $v$, a contradiction. Therefore, there is no good $(u,v)$-pair in $Q$ when (ii) or (iii) holds.

	Now suppose that none of (i)-(iii) holds. We proceed to prove that there is a good $(u,v)$-pair in $Q$. First we may assume that $Q$ is not 2-arc-strong, since otherwise there is nothing to prove by Lemma \ref{2arcstrong} and the fact that $Q$ is not isomorphic to one of the two digraphs in Table \ref{exceptions} (a).  Further, by Lemmas \ref{goodinD} and \ref{uneqvsameQ}, we may assume that there is no good $(u_S,v_S)$-pair in $S$, where $u_S$ and $v_S$ are the vertices in $S$ which  $u$ and $v$ correspond to, respectively. Recall that by Remark \ref{Remark-uSvS},  $u_S=v_S=u$ when $H(u)=H(v)$ and $u\neq v$ and, $u_S=u,v_S=v$ when $H(u)\neq H(v)$ or $u=v$. If the digraph $S$ is isomorphic to one of the digraphs in Figure \ref{fig-SDexceptions} (c)-(f), then either its composition $Q$ (or $\overleftarrow{Q}$) is isomorphic to one of the digraphs in Table \ref{exceptions} (b), (c), (g) or there is a good $(u,v)$-pair in $Q$. So we may assume that $S$ is not isomorphic to any such digraph. Then by the fact that $S$ has no good $(u_S,v_S)$-pair, there is a partition of $V(S)$ satisfying Lemma \ref{noGP}. Let $(V_1,\ldots,V_p)$ be such a partition and let  $(x_1y_1,\ldots,x_qy_q)$ be the ordering of the backward arcs such that $x_1\in V_p$ and $y_q\in V_1$. Now we are ready to apply Lemma \ref{noGP}.
	
	First we consider the case that $(S,u_S,v_S)$ is either of type A or of type B with $\beta=1$. In this case, $S$ has a $(u_S,v_S)$-pair where the branchings share only one arc. Since conditions (ii) and (iii) do not hold, we may assume that there is a backward arc $xy$ of $S$ with  $|H(x)|\geq 2$ such that there is a vertex $w\in H(x)$ with $d_Q^+(w)\geq 2$. 
	Then by Lemmas \ref{uneqvsameQ1} and \ref{shareonearc} (or by Corollary \ref{shareonearcsymmetry} if we consider the case $|H(y)|\geq 2$),  we have that $Q$ has a good $(u,v)$-pair, otherwise, the statement (i) of the theorem holds and we obtain a contradiction.
		
	So it suffices to consider the case that every pair of  out- and in-branchings $B_{u_S,S}^+, B_{v_S,S}^-$ of $S$ share at least two arcs, that is, we are in case (II) of Lemma \ref{noGP}. This implies that $(S,u_S,v_S)$ is of type B with $p=\beta+1, q=\beta, \beta\geq 2$. Since (iii) does not hold,  either $|H(x_i)|\geq 2$ or $|H(y_i)|\geq 2$, for each backward arc $x_iy_i$.
	
	Starting from the semicomplete digraph $S$ we now construct the induced subdigraph $Q^{\prime}$ of $Q$ as follows: 
	
	\textbf{(a)} For the backward arc $x_1y_1$, if $|H(y_1)|\geq 2$, then add a copy of $y_1$ to $V(Q^{\prime})$, otherwise add a copy of $x_1$ to $V(Q^{\prime})$.
          
        \textbf{(b)} For the backward arc $x_{\beta}y_{\beta}$, if $|H(x_{\beta})|\geq 2$, then add a copy of $x_{\beta}$ to $V(Q^{\prime})$, otherwise add a copy of $y_{\beta}$. 
	
      \textbf{(c)} For any $x_iy_i$ with $1<i<\beta$, add all vertices of $H(x_i)$ and $H(y_i)$ to $V(Q^{\prime})$. \\

	In the following, we use $x_i^{\prime}$ and $y_i^{\prime}$ to denote a copy of $x_i$ and $y_i$, that is, a vertex of $H(x_i)-x_i$ and $H(y_i)-y_i$, respectively. It follows by Remark \ref{Remark-copyofD} that for a vertex $w$, every  vertex in $H(w)$ can be regarded as the vertex $w$ or $w^{\prime}$. So if there is an arc $ab$ in $H(w)$, we choose $\{w,w^{\prime}\}=\{a,b\}$ and we may choose $a$ (resp., $b$) to be the vertex $w$ or $w^{\prime}$ if the arc $ab$ is useful as an out-arc of $a$ (resp., in-arc of $b$) in the argument below.
	
	Next we construct a good $(u,v)$-pair in $Q^{\prime}$ and then it follows by Lemma \ref{(D-X)->D} that $Q$ has the wanted pair. Since there is a pair of branchings $B_{u,S}^+,B_{v,S}^-$ in $S$ such that $A(B_{u,S}^+)\cap A(B_{v,S}^-)=\{x_1y_1,\ldots,x_{\beta}y_{\beta}\}$ (Lemma \ref{noGP} (II)), the subdigraph  $S\left\langle V_1\right\rangle$ has an out-branching $O$ rooted at $y_{\beta}$ which arc-disjoint from some $(y_{\beta},v)$-path $P_{y_{\beta},v}$. Similarly, the subdigraph $S\left\langle V_{\beta+1}\right\rangle$ 
contains an in-branching $I$ rooted at $x_1$ and a path  $P_{u,x_1}$ which are arc-disjoint.
Since $Q^{\prime}$ is a composition of $S$, there is an out-branching $O^{\prime}$ rooted at $y_{\beta}^{\prime}$ in $Q^{\prime}\left\langle(V_1-y_{\beta})\cup \{y_{\beta}^{\prime}\}\right\rangle$ which is arc-disjoint from the path $P_{y_{\beta},v}$ and 
and an in-branching $I^{\prime}$ in $Q^{\prime}\left\langle(V_{\beta+1}-x_1)\cup \{x_1^{\prime}\}\right\rangle$ which is arc-disjoint from the path $P_{u,x_1}$.  
Recall that there are two arc-disjoint $(y_{i-1},x_i)$-paths in $S\left\langle V_{\beta+2-i}\right\rangle$  when $y_{i-1}\neq x_i$ (the definition of type B in Definition \ref{typeAB}). By the way we defined $Q^{\prime}$ it follows that there are two arc-disjoint paths from $H(x_1)$ to $H(y_{\beta})$ in $Q^{\prime}$  as either $|H(x_i)|\geq 2$ or $|H(y_i)|\geq 2$ for each $x_iy_i$. Next we show how to construct a good $(u,v)$-pair in $Q^{\prime}$ from such a pair of arc-disjoint paths. 
	\medskip

\textbf{Case 1.} $|H(y_1)|=|H(x_{\beta})|=1$.
	
In this case the vertices $x_1,x_1^{\prime},y_{\beta}$ and $y_{\beta}^{\prime}$ belong to $V(Q^{\prime})$. Since (iii) does not hold, we may assume that $x_1$ has out-degree at least two and $y_{\beta}$ has in-degree at least two in $Q$. Let $P_1$ be an $(x_1,y_{\beta}^{\prime})$-path and let $P_2$ be an $(x_1^{\prime},y_{\beta})$-path in $Q^{\prime}$ such that they are arc-disjoint. As $|H(y_1)|=|H(x_{\beta})|=1$, we have that $y_1$ and $x_{\beta}$ belong to both $P_1$ and $P_2$. Using $P_1$ and $P_2$ one can construct a good $(u,v)$-pair $(B_{u}^+,B_{v}^-)$ in $Q^{\prime}$ as follows. Construct $B_{u}^+$ from the paths $P_{u,x_1}$, $P_1$ and the out-branching $O^{\prime}$ rooted at $y_{\beta}^{\prime}$ in $Q\left\langle(V_1-y_{\beta})\cup \{y_{\beta}^{\prime}\}\right\rangle$ by adding an in-arc of $y_{\beta}$ which is distinct from $x_{\beta}y_{\beta}$ and all arcs from $y_{\beta}$ to uncovered vertices of $V(Q^{\prime})$. By symmetry, one can construct $B_{v}^-$
from the paths $P_{y_{\beta},v}$, $P_2$ and the in-branching $I^{\prime}$ rooted at $x_1^{\prime}$ in $Q\left\langle(V_{\beta+1}-x_1)\cup \{x_1^{\prime}\}\right\rangle$ by adding an out-arc of $x_1$ which is distinct from $x_1y_1$ and all arcs from  uncovered vertices of $V(Q^{\prime})$ to $x_1$. See Figure \ref{fig:Case1}.
        
        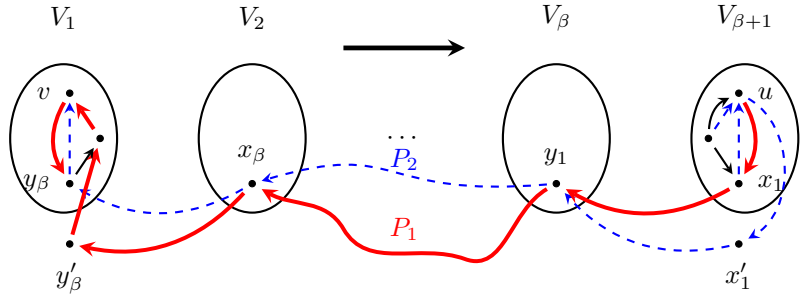
\begin{figure}[H]
				\centering\begin{tikzpicture}[scale=0.4]
					\foreach \i in {(11.2,0),(5,0),(-5,0),(-11.2,0)}{\draw[ line width=0.8pt] \i ellipse [x radius=50pt, y radius=70pt];}
					\coordinate [label=center:$V_1$] () at (-11.2,4);
					\coordinate [label=center:$V_2$] () at (-5,4);
					\coordinate [label=center:$\cdots$] () at (0,0);
					\coordinate [label=center:{\color{red}$P_1$}] () at (0,-3);
					\coordinate [label=center:{\color{blue}$P_2$}] () at (0,-0.7);
					\coordinate [label=center:$V_\beta$] () at (5,4);
					\coordinate [label=center:$V_{\beta+1}$] () at (11.2,4);
					\draw[-stealth,line width=1.8pt] (-2,3) -- (2,3); 
					\filldraw[black](11,-1.5) circle (3pt)node[label=right:$x_1$](x1){};
					\filldraw[black](-11,-1.5) circle (3pt)node[label=left:$y_\beta$](yb){};
					\filldraw[black](-11,1.5) circle (3pt)node[label=left:$v$](v){};
					\filldraw[black](-5,-1.5) circle (3pt)node[label=above:$x_\beta$](xb){};
					\filldraw[black](11,1.5) circle (3pt)node[label=right:$u$](u){};
					\filldraw[black](5,-1.5) circle (3pt)node[label=above:$y_1$](y1){};
					\filldraw[black](11,-3.5) circle (3pt)node[label=below:$x_1^{\prime}$](x2){};
					\filldraw[black](-11,-3.5) circle (3pt)node[label=below:$y_\beta^{\prime}$](y2){};
					\filldraw[black](-10,0) circle (3pt)node(a){};
					\filldraw[black](10,0) circle (3pt)node(b){};
					
					\foreach \i/\j/\c/\t/\a in {
						x1/u/blue/0/0.8,
						b/u/blue/0/0.8,
						yb/v/blue/0/0.8,
						u/x2/blue/60/0.8,
						x2/y1/blue/30/0.8,
						xb/yb/blue/30/0.8
					}{\path[draw,dashed, \c, line width=\a] (\i) edge[bend left=\t] (\j);}	
				\foreach \i/\j/\c/\t/\a in {
					y2/a/red/0/1.5,
					a/v/red/0/1.5,
					v/yb/red/-30/1.5,
					u/x1/red/30/1.5,
					x1/y1/red/30/1.5,
					xb/y2/red/30/1.5,
					b/x1/black/0/0.8,
					yb/a/black/0/0.8,
					b/u/black/30/0.8
				}{\path[draw, \c, line width=\a] (\i) edge[bend left=\t] (\j);}	
					\draw [red,line width=1.5pt] (y1)
					to [out=210,in=-20] (2,-4) to [out=-200,in=-60] (-2,-3) to [out=120, in=-30] (xb);
					\draw [blue, dashed, line width=0.8pt] (y1) .. controls (-1,-2) and (0,0) .. (xb);
                                      \end{tikzpicture}\caption{Illustrating Case 1. The vertex $y_{\beta}$ is used to collect those vertices of $V(Q^{\prime})-V_1$ which are not covered by the paths $P_1,P_{u,x_1}$ 
                                        Similarly $x_1$ is used to collect those vertices of $V(Q^{\prime})-V_{\beta+1}$ which are not covered by the paths $P_2,P_{y_{\beta},v}$.}\label{fig:Case1}
                                  \end{figure}

        Before considering Cases 2 and 3, we claim that if $|H(y_{i-1})|\geq 2$ for some $2\leq i\leq \beta$, then there is a vertex in $H(y_{i-1})$ with out-degree at least two in $Q$ (and thus also in $Q^{\prime}$ by the definition of $Q^{\prime}$). Recall that there are two arc-disjoint $(y_{i-1},x_i)$-paths in $S\left\langle V_{\beta+2-i}\right\rangle$  when $y_{i-1}\neq x_i$. So if $y_{i-1}\neq x_i$, there is nothing to prove. Therefore, we may assume that $y_{i-1}=x_i$ and then $|H(x_i)|=|H(y_{i-1})|\geq 2$. The claim follows by the assumption that the statement (iii) of Theorem \ref{mainthm} does not hold. It should be noted that if all out-neighbors of $y_1$ in $Q-y_2$ are in $H(y_1)$, then we will pick one of such out-neighbor as $y_{1}^{\prime}$, so $y_1$ also has out-degree two in $Q^{\prime}$. By symmetry, if $|H(x_{i})|\geq 2$ for some $2\leq i\leq \beta$, then there is a vertex in $H(x_{i})$ with in-degree at least two in $Q$.
	
	\textbf{Case 2.} Exactly one of $|H(y_1)|\geq 2$ and $|H(x_{\beta})|\geq 2$ holds. 
	
	Reversing all arcs of $Q^{\prime}$ and switching the role of $u$ and $v$ if necessary, we may assume that $|H(y_1)|\geq 2$ and $|H(x_{\beta})|=1$. By the construction of $Q^{\prime}$, we have that  $y_1^{\prime}$ and $y_{\beta}^{\prime}$ belong to $V(Q^{\prime})$. Again, since (iii) does not hold, we may assume that $y_{\beta}$ has in-degree at least two in $Q$. By the claim before Case 2, we may assume that $y_1$ has out-degree at least two in $Q^{\prime}$. Let $P_1$ (resp., $P_2$) be an $(x_1,y_{\beta}^{\prime})$-path containing $y_1$ (resp., an $(x_1,y_{\beta})$-path containing $y_1^{\prime}$)  such that they are arc-disjoint and let $e$ be an out-arc of $y_1$ in $Q^{\prime}$ which is not on $P_1$.  It follows by $|H(x_{\beta})|=1$ that $x_{\beta}\in P_1\cap P_2$. 
	
	Then one can construct a good $(u,v)$-pair $(B_{u}^+,B_{v}^-)$ in $Q^{\prime}$ as follows (similar to the construction in Case 1). Construct $B_{u}^+$ from the paths $P_{u,x_1}$, $P_1$ and the out-branching $O^{\prime}$ rooted at $y_{\beta}^{\prime}$ in $Q\left\langle(V_1-y_{\beta})\cup \{y_{\beta}^{\prime}\}\right\rangle$ by adding an in-arc of $y_{\beta}$ which is distinct from $x_{\beta}y_{\beta}$ and all arcs from $y_{\beta}$ to uncovered vertices of $V(Q^{\prime})$. By symmetry, one can construct $B_{v}^-$ the paths $P_{y_{\beta},v}$, $P_2$ and the in-branching $I$ rooted at $x_1$ in $Q\left\langle V_{\beta+1}\right\rangle$ by adding the arc $e$ and all arcs from  uncovered vertices of $V(Q^{\prime})$ to $x_1$. See Figure \ref{fig:Case2}.
        
        Here, it should be noted that $y_1$ may not dominate $x_1$, so we use the out-arc $e$ of $y_1$ to collect $y_1$ into $B_{v}^-$.

\begin{figure}[H]
				\centering\begin{tikzpicture}[scale=0.4]
					\foreach \i in {(11.2,0),(5,0),(-5,0),(-11.2,0)}{\draw[ line width=0.8pt] \i ellipse [x radius=50pt, y radius=70pt];}
					\coordinate [label=center:$V_1$] () at (-11.2,4);
					\coordinate [label=center:$V_2$] () at (-5,4);
					\coordinate [label=center:$\cdots$] () at (0,0);
					\coordinate [label=center:{\color{red}$P_1$}] () at (0,-2.5);
					\coordinate [label=center:{\color{blue}$P_2$}] () at (0,-1);
					\coordinate [label=center:$V_\beta$] () at (5,4);
					\coordinate [label=center:$V_{\beta+1}$] () at (11.2,4);
					\draw[-stealth,line width=1.8pt] (-2,3) -- (2,3); 
					\filldraw[black](11,-1.5) circle (3pt)node[label=right:$x_1$](x1){};
					\filldraw[black](-11,-1.5) circle (3pt)node[label=left:$y_\beta$](yb){};
					\filldraw[black](-11,1.5) circle (3pt)node[label=left:$v$](v){};
					\filldraw[black](-5,-1.5) circle (3pt)node[label=above:$x_\beta$](xb){};
					\filldraw[black](11,1.5) circle (3pt)node[label=right:$u$](u){};
					\filldraw[black](5,-1.5) circle (3pt)node[label=above:$y_1$](y1){};
					\filldraw[black](-11,-3.5) circle (3pt)node[label=below:$y_\beta^{\prime}$](y2){};
					\filldraw[black](5,-3.5) circle (3pt)node[label=below:$y_1^{\prime}$](y3){};
					\filldraw[white](6,0) circle (3pt)node[](y4){};
					\filldraw[black](-10,0) circle (3pt)node(a){};
					\filldraw[black](10,0) circle (3pt)node(b){};
					
					\foreach \i/\j/\c/\t/\a in {
						x1/u/blue/0/0.8,
						b/u/blue/0/0.8,
						yb/v/blue/0/0.8,
						x1/y3/blue/30/0.8,
						xb/yb/blue/30/0.8,
						y1/y4/blue/0/0.8
					}{\path[draw, dashed,\c, line width=\a] (\i) edge[bend left=\t] (\j);}
				\foreach \i/\j/\c/\t/\a in {
					y2/a/red/0/1.5,
					a/v/red/0/1.5,
					v/yb/red/-30/1.5,
					u/x1/red/30/1.5,
					x1/y1/red/30/1.5,
					xb/y2/red/30/1.5,
					b/x1/black/0/0.8,
					yb/a/black/0/0.8,
					b/u/black/30/0.8
				}{\path[draw, \c, line width=\a] (\i) edge[bend left=\t] (\j);}		
					\draw [red,line width=1.5pt] (y1)
					to [out=210,in=-20] (1,-3) to [out=-200,in=-60] (-2,-3) to [out=120, in=-30] (xb);
					\draw [blue, dashed, line width=0.8pt] (y3) .. controls (-1,-2) and (0,0) .. (xb);	
                                      \end{tikzpicture}\caption{Case 2}\label{fig:Case2}
                                    \end{figure}
	\medskip

	\textbf{Case 3.} $|H(y_1)|\geq 2$ and $|H(x_{\beta})|\geq 2$. 
	
	In this case, it follows from the way we defined $Q^{\prime}$ that $x_1$ and $y_{\beta}$ have no copy in $Q^{\prime}$. By the claim before Case 2, we may assume that $y_1$ has out-degree at least two and $x_{\beta}$ has in-degree at least two in $Q^{\prime}$.  Let $P_1$ and $P_2$ be a pair of arc-disjoint $(x_1,y_{\beta})$-paths in $Q^{\prime}$ such that $y_1\in P_1$ and $x_{\beta}\in P_2$. Let $e_y$ be any out-arc of $y_1$ in $Q^{\prime}$ which is not on $P_1$ and let $e_x$ be any in-arc of $x_{\beta}$ in $Q^{\prime}$ which is not on the path $P_2$. 
	
	If $e_x\neq e_y$, then one can construct a good $(u,v)$-pair $(B_{u}^+,B_{v}^-)$ in $Q^{\prime}$ as follows (similar to the construction in Case 1). Construct $B_{u}^+$ from the paths $P_{u,x_1}$, $P_1$ and the out-branching $O$ rooted at $y_{\beta}$ in $Q\left\langle V_1 \right\rangle$ by adding the arc $e_x$ and all arcs from $y_{\beta}$ to uncovered vertices of $V(Q^{\prime})$. By symmetry, one can construct $B_{v}^-$ the paths $P_{y_{\beta},v}$, $P_2$ and the in-branching $I$ rooted at $x_1$ in $Q\left\langle V_{\beta+1}\right\rangle$ by adding the arc $e_y$ and all arcs from  uncovered vertices of $V(Q^{\prime})$ to $x_1$. See Figure \ref{fig:Case3}.

        \begin{figure}[H]
				\centering\begin{tikzpicture}[scale=0.4]
					\foreach \i in {(11.2,0),(5,0),(-5,0),(-11.2,0)}{\draw[ line width=0.8pt] \i ellipse [x radius=50pt, y radius=70pt];}
					\coordinate [label=center:$V_1$] () at (-11.2,4);
					\coordinate [label=center:$V_2$] () at (-5,4);
					\coordinate [label=center:$\cdots$] () at (0,0);
					\coordinate [label=center:{\color{red}$P_1$}] () at (0,-3);
					\coordinate [label=center:{\color{blue}$P_2$}] () at (0,-1);
					\coordinate [label=center:$V_\beta$] () at (5,4);
					\coordinate [label=center:$V_{\beta+1}$] () at (11.2,4);
					\draw[-stealth,line width=1.8pt] (-2,3) -- (2,3); 
					\filldraw[black](11,-1.5) circle (3pt)node[label=right:$x_1$](x1){};
					\filldraw[black](-11,-1.5) circle (3pt)node[label=left:$y_\beta$](yb){};
					\filldraw[black](-11,1.5) circle (3pt)node[label=left:$v$](v){};
					\filldraw[black](-5,-1.5) circle (3pt)node[label=above:$x_\beta$](xb){};
					\filldraw[black](-5,-3.5) circle (3pt)node[label=below:$x_\beta^{\prime}$](x2){};
					\filldraw[black](11,1.5) circle (3pt)node[label=right:$u$](u){};
					\filldraw[black](5,-1.5) circle (3pt)node[label=above:$y_1$](y1){};
					\filldraw[black](5,-3.5) circle (3pt)node[label=below:$y_1^{\prime}$](y2){};
					\filldraw[white](-6.5,0) circle (3pt)node[](y3){};
					\filldraw[white](6,0) circle (3pt)node[](y4){};
					\filldraw[black](-10,0) circle (3pt)node(a){};
					\filldraw[black](10,0) circle (3pt)node(b){};
					
					\foreach \i/\j/\c/\t/\a in {
						x1/u/blue/0/0.8,
						b/u/blue/0/0.8,
						yb/v/blue/0/0.8,
						x1/y2/blue/30/0.8,
						xb/yb/blue/30/0.8,
						y1/y4/blue/0/0.8
					}{\path[draw, dashed,\c, line width=\a] (\i) edge[bend left=\t] (\j);}	
					\foreach \i/\j/\c/\t/\a in {
					a/v/red/0/1.5,
					y3/xb/red/0/1.5,
					v/yb/red/-30/1.5,
					u/x1/red/30/1.5,
					x1/y1/red/30/1.5,
					x2/yb/red/30/1.5,
					b/x1/black/0/0.8,
					yb/a/black/0/0.8,
					b/u/black/30/0.8
				}{\path[draw, \c, line width=\a] (\i) edge[bend left=\t] (\j);}	
					\draw [red, line width=1.5pt] (y1)
					to [out=210,in=-20] (2,-4) to [out=-200,in=-60] (-2,-3) to [out=120, in=30] (x2);
					\draw [blue, dashed, line width=0.8pt] (y2) .. controls (-1,-2) and (0,0) .. (xb);	
				\end{tikzpicture}\caption{Case 3}\label{fig:Case3}
                            \end{figure}

                            So it suffices to consider the case that $y_1$ has only one out-arc $e_y$ not on $P_1$ and $x_{\beta}$ has only one in-arc $e_x$ not on $P_2$ and, these two arcs are the same one. Recall that $(S,u,v)$ is of type B and each backward arc $x_iy_i$ goes from $V_{\beta+2-i}$ to $V_{\beta+1-i}$ (two consecutive sets). So  $2\leq \beta\leq 3$. Further, we must have $x_i=y_{i-1}$ for $i=2$ when $\beta=2$ and for each $i\in\{2,3\}$ when $\beta=3$. Otherwise, for $i=2$ there are two arc-disjoint $(y_1,x_2)$-paths $P,P^{\prime}$ in $S\left\langle V_{\beta}\right\rangle$. Assume that $P$ is shorter than $P^{\prime}$ and $P$ is contained in the path $P_1$. This implies that $y_1$ has an out-neighbor which is an internal vertex of $P^{\prime}$, say $w$ is such out-neighbor. Then $y_1$ has an out-arc $y_1w\neq y_1x_{\beta}$ which is not on $P_1$, a contradiction. By symmetry, for $i=3$, we have $y_2=x_3$. So we may assume that $\{x_1y_1,y_1y_2\}$ and   $\{x_1y_1,y_1y_2,y_2y_3\}$ are the sets of backward arcs for the cases $\beta=2$ and $\beta=3$, respectively. Further, by Remark \ref{Remark-copyofD}, we may assume that $x_{\beta}=y_{\beta-1}^{\prime}$. Thus $P_1=uy_1v,P_2=uy_1^{\prime}v$ if $\beta=2$ and $P_1=uy_1y_2v,P_2=uy_1^{\prime}y_2^{\prime}v$ if $\beta=3$, see Figure \ref{fig-main23}.

		\begin{figure}[!h]
			\subfigure{\begin{minipage}[t]{0.45\linewidth}
					\centering\begin{tikzpicture}[scale=0.5]
					\filldraw[black](0,0) circle (3pt)node[label=left:$y_1$](y1){};
					\filldraw[black](-2,3.76) circle (3pt)node[label=left:$v$](v){};
					\filldraw[black](-2,-3.76) circle (3pt)node[label=left:$u$](u){};
					\filldraw[black](2,0) circle (3pt)node[label=below:$y_1^{\prime}$](y2){};
					\foreach \i/\j in {u/y1,y1/v}{\path[draw, red, line width=1.8pt] (\i) edge (\j);}
					\foreach \i/\j in {y1/y2,v/u}{\path[draw, line width=0.8pt] (\i) edge (\j);}
					\foreach \i/\j in {u/y2,y2/v}{\path[draw, blue, dashed, line width=0.8pt] (\i) edge (\j);}
					\end{tikzpicture}\caption*{(a) $\beta=2$}\end{minipage}}~~~~~~~~
			\subfigure{\begin{minipage}[t]{0.45\linewidth}
					\centering\begin{tikzpicture}[scale=0.5]
					\foreach \i in {(1,1.5),(1,-1.5)}{\draw[ line width=0.8pt] \i ellipse [x radius=60pt, y radius=30pt];}
					\draw[-stealth,line width=1.5pt] (-2,1.1) -- (-2,-1.1); 
					\filldraw[black](1.5,-1.5) circle (3pt)node[label=right:$y_1$](y1){};
					\filldraw[black](1.5,1.5) circle (3pt)node[label=right:$y_2$](y2){};
					\filldraw[black](1.5,4) circle (3pt)node[label=left:$v$](v){};
					\filldraw[black](1.5,-4) circle (3pt)node[label=left:$u$](u){};
					\filldraw[black](4,-1.5) circle (3pt)node[label=right:$y_1^{\prime}$](y3){};
					\filldraw[black](4,1.5) circle (3pt)node[label=right:$y_2^{\prime}$](y4){};
					\coordinate [label=center:$V_3$] () at (0,-1.5);
					\coordinate [label=center:$V_2$] () at (0,1.5);
					\foreach \i/\j in {u/y1,y1/y2,y2/v}{\path[draw, red, line width=1.8pt] (\i) edge (\j);}
					\foreach \i/\j in {u/y3,y3/y4,y4/v}{\path[draw, blue, dashed, line width=0.8pt] (\i) edge (\j);}
					\foreach \i/\j in {y3/y2,y1/y4}{\path[draw, line width=0.8pt] (\i) edge (\j);}
					\end{tikzpicture}\caption*{(b) $\beta=3$}\end{minipage}}
			\caption{The red fat arcs and blue dashed arcs indicate the paths $P_1$ and $P_2$, respectively.}\label{fig-main23}
		\end{figure}

For the case $\beta=2$, since $y_1$ has only one out-arc not on $P_1$ and $x_{\beta}$ has only one in-arc not on $P_2$, we have that $V_1=\{v\}, V_2=\{y_1\}, V_3=\{u\}$ in $S$. As $Q$ is not isomorphic to the digraph in Table \ref{exceptions} (c), either $H(y_1)$ has at least two arcs or one of $|H(u)|\geq 2$ and $|H(v)|\geq 2$ holds (in $Q$).  For the former case, let $a_ib_i,i\in[2]$ be such two arcs, then for paths $P_1^{\prime}=ua_1v$ and $P_2^{\prime}=ub_2v$, there is an out-arc $a_1b_1$ of $a_1$ not on $P_1^{\prime}$ and an in-arc $a_2b_2$ of $b_2$ not on $P_2^{\prime}$ and they are distinct. Then we may obtain a good $(u,v)$-pair by considering $P_i^{\prime}$ rather than $P_i$, $i\in[2]$. So we only need to consider that the latter case holds and assume that $|H(u)|\geq 2$. Then $uy_1vu^{\prime}y_1^{\prime}$ and $u^{\prime}y_1y_1^{\prime}v\cup\{uy_1^{\prime}\}$ form a good $(u,v)$-pair of $Q\left\langle \{u,u^{\prime},v,y_1,y_1^{\prime}\} \right\rangle$ and hence $Q$ has the wanted pair by Lemma \ref{(D-X)->D}.
                
                For the case $\beta=3$, since $y_1$ has only one out-arc not on $P_1$ and $x_{\beta}$ has only one in-arc not on $P_2$, again we have $V_1=\{y_3\}=\{v\},V_4=\{x_1\}=\{u\}$ in the partition of $S$. Moreover, both $H(y_1)$ and $H(y_2)$ are independent sets of size at least two. Suppose that there is a vertex $w\in V(Q)$ not in $H(y_1)\cup H(y_2)\cup H(u)\cup H(v)$.  Since  $y_1y_2^{\prime}$ is the only out-arc of $y_1$ not on $P_1$ and the only in-arc of $y_2^{\prime}$ not on $P_2$, we have that $H(y_2)$ dominates $w$ and $w$ dominates $H(y_1)$. Then $uy_1y_2wy_1^{\prime}y_2^{\prime}v$ and $y_1y_2^{\prime}wuy_1^{\prime}y_2v$ form a good $(u,v)$-pair in the subdigraph $Q\left\langle V(Q^{\prime})\cup\{w\}\right\rangle$ and thus $Q$ has the wanted pair by Lemma \ref{(D-X)->D}. So $Q=C_3[H(u)\cup H(v),H(y_1),H(y_2)]$, which  contradicts with our assumption (in the beginning of the proof) that $Q$ is not 2-arc-strong. This completes the proof.\end{proof}

                Gutin and Sun proved the following structural property of good $(u,u)$-pairs in semicomplete compositions.

\begin{lem}\label{QvsQprime} \cite{gutinDM343}
	A strong semicomplete composition $Q$ has a good pair rooted at $v$ if and only if $Q^{\prime}=Q\left\langle\{v\}\cup N^-(v)\cup N^+(v)\right\rangle$ has a good pair rooted at $v$.
\end{lem}
 Next we derive a result similar to Lemma \ref{QvsQprime} for distinct roots $u$ and $v$.
\begin{lem} 
	Let $S$ be a strong semicomplete digraph of order $s\geq 2$ and let $H_1,\ldots,H_s$ be arbitrary digraphs. Suppose that $Q=S[H_1,\ldots,H_s]$ and $u, v$ are two distinct vertices in the same $H_i$. Then  $Q^{\prime}=Q\left\langle\{u,v\}\cup N(u)\cup N(v)\right\rangle$  has a good $(u,v)$-pair if and only if $Q$ has a good $(u,v)$-pair and $Q^{\prime}$ is not isomorphic to the digraph  $C_3[\overline{K}_2,\overline{K}_2,H]$, where $V(H)=\{u,v\}$ and $A(H)\subseteq\{vu\}$, that is, one of the two digraphs in Table \ref{exceptions} (a).
\end{lem}

\begin{proof}
  Note that $V(Q)\backslash V(Q^{\prime})\subseteq V(H_i)$, which means that $Q^{\prime}$ can be obtained from $Q$ by deleting those vertices in $H_i-\{u,v\}$ which are not adjacent to $u$ and $v$. Then $d^+_{Q^{\prime}}(u)=d^+_{Q}(u)$ and $d^-_{Q^{\prime}}(v)=d^-_{Q}(v)$. First suppose that $Q^{\prime}$ has a good $(u,v)$-pair. By Lemma \ref{specialexceptions}, $Q^{\prime}$ is not isomorphic to the exception.
  Since $Q$ is a composition of the strong digraph $S$, each vertex in $H_i$ has an in- and out-neighbor in $V(Q-H_i)$ and hence it follows from Lemma \ref{(D-X)->D} that there is a good $(u,v)$-pair in $Q$.
	
	Next we show the sufficiency. Suppose that $Q^{\prime}$ has no good $(u,v)$-pair. Since $u$ and $v$ are two vertices in the same $H_i$, the vertices $u$ and $v$ have the same out- and in-neighbors in $H_j$ for each $j\neq i$ and $u_S=v_S$.  By the assumption that $Q^{\prime}$ has no good $(u,v)$-pair and it is not isomorphic to the exception, $(S,u_S,v_S)$ is of type A with $\alpha=1$ and the second statement of Theorem \ref{mainthm} holds for $Q^{\prime}$. Recall that $Q^{\prime}$ can be obtained from $Q$ by deleting some vertices in $H_i$. Then the statement (ii) also holds for $Q$ and then there is no good $(u,v)$-pair in $Q$, which contradicts our assumption. Hence there is a good $(u,v)$-pair in $Q^{\prime}$ and then the lemma holds.
\end{proof}

\section{Good pairs in quasi-transitive  digraphs and compositions of transitive digraphs }\label{sec:transitive}

\begin{lem} \label{transitive}\cite{bang2009}
	Let $D$ be a digraph with an acyclic ordering $D_1,\ldots,D_p$ of its strong components. The digraph $D$ is transitive if and only if each $D_i$ is complete, the digraph $H$ obtained from $D$ by contracting of $D_1,\ldots,D_p$ followed by deletion of multiple arcs is a transitive oriented graph, and $D=H[D_1,\ldots,D_p]$, where $p=|V(H)|$.
\end{lem}

\begin{thm}\label{Dtransitive}
  Let $T$ be a transitive digraph on $t\geq 2$ vertices and let $H_1,\ldots,H_t$ be arbitrary digraphs. Suppose that $Q=T[H_1,\ldots,H_t]$ and $u$ (resp., $v$) is a vertex in the initial (resp., terminal) component of $Q$. Then there is a good $(u,v)$-pair in $Q$ if and only if none of the following holds.
    \begin{itemize}
    \item $Q=TT_3[u,\overline{K}_{n-2},v]$ (the digraph in Table \ref{exceptions} (b)).
    \item $Q=C_2[H,K_1]$ such for some digraph $H$ so that $u$ and $v$ are two distinct vertices of $H$ and $d^+_Q(u)=1$ or $d^-_Q(v)=1$.
    \item $Q=P_2[T^+_u,v]$ or $Q=P_2[u,T^-_v]$, where $T^+_u$ (resp., $T^-_v$) is an out-tree rooted at $u$ (resp., in-tree rooted at $v$). 
      \end{itemize}
\end{thm}

\begin{proof}
	Observe that if $Q$ is isomorphic to $P_2[T_u,v]$ or $P_2[u,T_v]$, then it clearly has no good $(u,v)$-pair as it has size $2|V(Q)|-3$ which is less than $2(|V(Q)|-1)$. Hence the necessity follows by Lemma \ref{specialexceptions} and Remark \ref{Remark-uvatleast2}.  Next we show the sufficiency. First consider the case that $T$ is strong and hence $T$ is complete by Lemma \ref{transitive}. If both $u$ and $v$ are vertices of $T$, there is a good $(u,v)$-pair in $T$ and hence $Q$ has a good $(u,v)$-pair due to Lemma \ref{(D-X)->D}. So it suffices to consider the case that $u$ and $v$ are two distinct vertices in the same $H_r$, that is, both $u$ and $v$ correspond to the same vertex of $T$. 
	
	If $t=2$, then $Q=C_2[H,H^{\prime}]$. By Lemmas \ref{Dstr-Qkstr} and \ref{2arcstrong}, we may assume that $H^{\prime}$ consists of a single vertex $w$ otherwise $Q$ is 2-arc-strong. As $u$ and $v$ are two distinct vertices in the same $H_r$, the vertices $u,v$ belong to $H$. Since $Q$ is not isomorphic to the exception, we get that $d^+_Q(u)\geq 2$ and $d^-_Q(v)\geq 2$. Let $u_O$ be an out-neighbor of $u$ and $v_I$ be an in-neighbor of $v$ in $H$. If $u$ dominates $v$, then $uvw$ and $uwv$ form a good $(u,v)$-pair in $Q\left\langle\{u,v,w\}\right\rangle$ and hence there is the wanted pair in $Q$ by Lemma \ref{(D-X)->D}. So we can assume that $uv\notin A(Q)$. Then we can get a good $(u,v)$-pair of $Q$ as follows. Let $B_{u,Q}^+=\{uw,v_Iv\}\cup\{wz:z\in V(H)-\{u,v\}\}$ and construct $B_{v,Q}^-$ from arcs $uu_O$ and $wv$ by adding arcs $\{zw:z\in V(H)-\{u,v\}\}$.  For the case that $t\geq3$, let $C=u_1u_2\cdots u_tu_1$ be a hamiltonian cycle in $T$ such that $u$ and $v$ correspond to the vertex $u_1$. By Remark \ref{Remark-copyofD}, we may  assume that $u$ is the vertex $u_1$ in $T$. Then $(C-u_tu_1)\cup \{u_tv\}$ and $u_1u_tu_{t-1}\cdots u_2v$ form a good $(u,v)$-pair in $Q\left\langle V(T)\cup\{v\}\right\rangle$. Lemma \ref{(D-X)->D} shows that $Q$ also has a good $(u,v)$-pair.

	It remains to consider the case that $T$ is a non-strong transitive digraph. In this case $u\neq v$ so we may assume that $u,v$ are vertices of $T$. Clearly, each  vertex in the initial component of $Q$ belong to a set $H(w)$ where  $w$ is a vertex of  the initial component of $T$. This means that there is a path from $u$ to any other vertex of $T$. As $T$ is transitive, $u$ dominates all vertices in $T-u$,  consequently, $u$ dominates all vertices in $Q-H(u)$. Moreover, since $u$ belongs to the initial component of $Q$, there is an out-branching $B_{u,H(u)}^+$ rooted at $u$ in $H(u)$. In the same way, each vertex in $Q-H(v)$ dominates $v$ and there is an in-branching $B_{v,H(v)}^+$ rooted at $v$ in $H(v)$.
	
	If $|H(u)|\geq 2$, then let $u^{\prime}$ be an arbitrary vertex in $H(u)-u$. Let  $B_{u,Q}^+=B_{u,H(u)}^+\cup \{ur:r\in Q-H(u)-v\}\cup\{u^{\prime}v\}$ and construct $I=B_{v,H(v)}^-\cup\{rv:r\in Q-H(v)-u^{\prime}\}$. Note that if $u^{\prime}$ has another out-arc $e$ in $Q-B_{u,H(u)}^+$ which is distinct from $u^{\prime}v$, then $(B_{u,Q}^+,I\cup \{e\})$ is a good $(u,v)$-pair in $Q$. So we may assume that if $|H(u)|\geq 2$, then for any $u^{\prime}\in H(u)-u$, it has only one out-arc $u^{\prime}v$ not in $B_{u,H(u)}^+$. This means that $v=Q-H(u)$ and thus $Q=P_2[H(u),v]$. Since $Q\neq P_2[T_u,v]$, there is an arc $ab\in H(u)$ not in $B_{u,H(u)}^+$, then $B_{u,H(u)}^+\cup\{av\}$ and $\{ab\}\cup\{rv:r\in H(u)-a\}$ form a good $(u,v)$-pair in $Q$. 
		
		Therefore, we may assume that $|H(u)|=1$ and assume $|H(v)|=1$ by symmetry. As $Q$ is not just the digraph on two vertices and the arc $uv$, i.e., $P_2[u,v]$, $Q$ contains a spanning subdigraph $TT_3[u,H,v]$, where is an arbitrary non-empty digraph. If there is an arc $xy$ with $x\in H\cup v$ and $y\in H$, then $\{ur:r\in Q-y\}\cup\{xy\}$ and $\{rv:r\in Q-u\}\cup\{uy\}$ form a good $(u,v)$-pair in $Q$. So we may assume that $H\cong \overline{K}_{n-2}$ and there is no arc from $v$ to $H$. By symmetry, there is no arc from $H$ to $u$ and then $Q$ is isomorphic to $TT_3[u,\overline{K}_{n-2},v]$, which contradicts our assumption. This completes the proof.
\end{proof}

Recall that in a quasi-transitive digraph, the presence of the arcs $xy,yz$ implies that there is an arc between $x$ and $z$. So if $Q=P_2[T_u,v]$ is quasi-transitive, where $T_u$ is an out-tree rooted at $u$, then either $T_u=u$ or $T_u=\{ur:r\in T_u-u\}$. Hence the following holds.
\begin{remark}\label{Remark-trancom}
	If $Q$ is a quasi-transitive digraph and it is isomorphic to $P_2[T_u,v]$ or $P_2[u,T_v]$ described in Theorem \ref{Dtransitive}, then $Q$ is isomorphic to $P_2[u,v]$ or $TT_3[u,\overline{K}_{n-2},v]$.
\end{remark}

Combining  Lemma \ref{QTchara}, Remark \ref{Remark-trancom}, Theorems \ref{Dtransitive} and \ref{mainthm} we obtain the following classification of good $(u,v)$-pairs in quasi-transitive digraphs.

\begin{thm}\label{mainthm-quasi}
	Let $Q$ be a quasi-transitive digraph. Suppose that $u$ (resp., $v$) is a vertex in the initial (resp., terminal) component of $Q$ such that $d_Q^+(u)\geq 2$ and $d_Q^-(v)\geq 2$ when $u\neq v$. Then $Q$ has a good $(u,v)$-pair if and only if $Q$ satisfies none of the conditions of Theorem \ref{mainthm}. 
\end{thm}

\section{Remarks}\label{sec:remarks}
Recall  Theorem \ref{thm:SDbranchchar}  which shows that there is a very
simple characterization for the existence of a good $(u,v)$-pair for every $u$ and $v$ in semicomplete digraphs (as mentioned earlier this is equivalent to Theorem \ref{SDgoodpair}).

As we point out below, no similar property holds in semicomplete compositions. In fact, even the following stronger condition is not sufficient to guarantee the existence of a good $(u,v)$-pair in a given semicomplete composition $Q$: For every choice of  vertices $w,z$ of $Q$, there is a $(w,v)$-path which is arc-disjoint from an out-branching rooted at $u$ in $Q$ and, there is a $(u,z)$-path which is arc-disjoint from an in-branching rooted at $v$ in $Q$. For example, let $S$ be a semicomplete digraph of type A with $\alpha=2, |V_1|=|V_5|=1$. Suppose that $Q$ is a composition of $S$ such that $H(x)=x$ for each $x\in S-\{x_1,y_3\}$ and, $V(H(y))=\{y,y^{\prime}\}, |A(H(y))|=0$ for $y\in\{x_1,y_3\}$, see Figure \ref{figremark}. The digraph $Q$ has no good $(u,v)$-pair due to Theorem \ref{mainthm} (ii). 
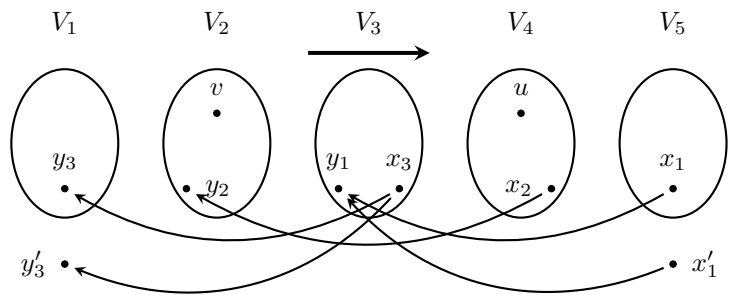
\begin{figure}[H]
	\centering\begin{tikzpicture}[scale=0.4]
	\foreach \i in {(-10,0),(-5,0),(0,0),(5,0),(10,0)}{\draw[ line width=0.8pt] \i ellipse [x radius=50pt, y radius=70pt];}
	\coordinate [label=center:$V_1$] () at (-10,4);
	\coordinate [label=center:$V_2$] () at (-5,4);
	\coordinate [label=center:$V_3$] () at (0,4);
	\coordinate [label=center:$V_4$] () at (5,4);
	\coordinate [label=center:$V_5$] () at (10,4);
	\draw[-stealth,line width=1.5pt] (-2,3) -- (2,3); 
	\filldraw[black](5,1) circle (3pt)node[label=above:$u$](u){};
	\filldraw[black](-5,1) circle (3pt)node[label=above:$v$](v){};
	\filldraw[black](-1,-1.5) circle (3pt)node[label=above:$y_1$](y1){};
	\filldraw[black](-6,-1.5) circle (3pt)node[label=right:$y_2$](y2){};
	\filldraw[black](-10,-1.5) circle (3pt)node[label=above:$y_3$](y3){};
	\filldraw[black](-10,-4) circle (3pt)node[label=left:$y_3^{\prime}$](y4){};
	\filldraw[black](10,-1.5) circle (3pt)node[label=above:$x_1$](x1){};
	\filldraw[black](6,-1.5) circle (3pt)node[label=left:$x_2$](x2){};
	\filldraw[black](1,-1.5) circle (3pt)node[label=above:$x_3$](x3){};
	\filldraw[black](10,-4) circle (3pt)node[label=right:$x_1^{\prime}$](x4){};
	
	\foreach \i/\j in {x1/y1,x2/y2,x3/y3}{\path[draw, line width=0.8pt] (\i) edge[bend left=30] (\j);}
	\foreach \i/\j in {x4/y1,x3/y4}{\path[draw, line width=0.8pt] (\i) edge[bend left=35] (\j);}
	\end{tikzpicture}\caption{An infinite class of semicomplete compositions $\mathcal{F}$ with no good $(u,v)$-pair but each $Q$ in the class has a branching $B_{v,Q}^-$ (resp., $B_{u,Q}^+$) which is arc-disjoint from a $(u,z)$-path (resp., $(w,v)$-path) for every choice of vertices $z,w$. The only arcs from right to left are the five arcs shown.}\label{figremark}
\end{figure}
Next we claim that $Q$ satisfies the above condition for any choice of vertices $z,w$ in $Q$. By symmetry, it suffices to show that there is a $(u,z)$-path which is arc-disjoint from an in-branching rooted at $v$ in $Q$. As $x_1$ and $x_1^{\prime}$ (resp., $y_3$ and $y_3^{\prime}$) are symmetric, we may assume that $z\notin\{x_1,y_3^{\prime}\}$. Construct a $(u,y_3)$-path $P$ from a $(u,x_2)$-path in $S\left\langle V_4 \right\rangle$ by adding arcs $\{x_2y_2,y_2x_3,x_3y_3\}$ and, construct $B_{v,Q}^-$ from a $(y_1,x_3)$-path in $S\left\langle V_3 \right\rangle$ by adding arcs $\{x_1y_1,x_1^{\prime}y_1, x_3y_3^{\prime},y_3^{\prime}v\}$ and all arcs from uncovered vertices to $x_1$. Then $P$ (or $P\cup \{y_3z\}$ if $z\notin P$) and $B_{v,Q}^-$ form a wanted pair. \\

\noindent{}{\bf Conflict of interest statement:}\\
There are no sources of conflict of interest regarding this paper.

\noindent{}{\bf Data availability statement:}\\
Data sharing is not applicable to this article as no data sets were generated or analyzed during the
current study.

\noindent{}{\bf Acknowledgments:}\\
Financial support from the Independent Research Fund Denmark under grant
DFF-7014-00037B is gratefully acknowledged. The second author was supported by China Scholarship
Council (CSC) No. 202106220108


\begin{thebibliography}{10}

\bibitem{bangJCT51}
J.~Bang-Jensen.
\newblock {Edge-disjoint in- and out-branchings in tournaments and related path
  problems}.
\newblock {\em J. Combin. Theory Ser. B}, 51(1):1--23, 1991.

\bibitem{bangJGT100}
J.~Bang{-}Jensen, S.~Bessy, F.~Havet, and A.~Yeo.
\newblock Arc-disjoint in- and out-branchings in digraphs of independence
  number at most 2.
\newblock {\em J. Graph Theory}, 100(2):294--314, 2022.

\bibitem{bangJGT28}
J.~Bang-Jensen and G.~Gutin.
\newblock {Generalizations of tournaments: A survey}.
\newblock {\em J. Graph Theory}, 28:171--202, 1998.

\bibitem{bang2009}
J.~Bang-Jensen and G.~Gutin.
\newblock {\em {Digraphs: Theory, Algorithms and Applications}}.
\newblock Springer-Verlag, London, 2nd edition, 2009.

\bibitem{bangJGT95}
J.~Bang{-}Jensen, G.~Gutin, and A.~Yeo.
\newblock Arc-disjoint strong spanning subdigraphs of semicomplete
  compositions.
\newblock {\em J. Graph Theory}, 95(2):267--289, 2020.

\bibitem{bangJGT20b}
J.~Bang-Jensen and J.~Huang.
\newblock {Quasi-transitive digraphs}.
\newblock {\em J. Graph Theory}, 20(2):141--161, 1995.

\bibitem{bangJGT42}
J.~Bang-Jensen, S.~Thomass{\'e}, and A.~Yeo.
\newblock {Small degree out-branchings}.
\newblock {\em J. Graph Theory}, 42(4):297--307, 2003.

\bibitem{bangJGTsub}
J.~Bang{-}Jensen and Y.~Wang.
\newblock Arc-disjoint in- and out-branchings in semicomplete digraphs.
\newblock {\em submitted, available on arXiv as arXiv:2302.06177}, 2023.

\bibitem{bangC24}
J.~Bang-Jensen and A.~Yeo.
\newblock {Decomposing {$k$}-arc-strong tournaments into strong spanning
  subdigraphs}.
\newblock {\em Combinatorica}, 24(3):331--349, 2004.

\bibitem{bercziIPL109}
K.~B{\'{e}}rczi, S.~Fujishige, and N.~Kamiyama.
\newblock A linear-time algorithm to find a pair of arc-disjoint spanning
  in-arborescence and out-arborescence in a directed acyclic graph.
\newblock {\em Inform. Process. Lett.}, 109(23-24):1227--1231, 2009.

\bibitem{edmonds1973}
J.~Edmonds.
\newblock {Edge-disjoint branchings}.
\newblock In {\em {Combinatorial Algorithms}}, pages 91--96. Academic Press,
  1973.

\bibitem{gutinDM343}
G.~Z. Gutin and Y.~Sun.
\newblock Arc-disjoint in- and out-branchings rooted at the same vertex in
  compositions of digraphs.
\newblock {\em Discret. Math.}, 343(5):111816, 2020.

\bibitem{lovaszJCT21}
L.~Lov{\'a}sz.
\newblock {On two min--max theorems in graph theory}.
\newblock {\em J. Combin. Theory Ser. B}, 21:96--103, 1976.

\end{thebibliography}

\end{document}